\documentclass[11pt,reqno]{elsarticle} 
\usepackage{color}
\usepackage[colorlinks=true, linkcolor=black, citecolor=black, urlcolor=black]{hyperref}
\usepackage{amsmath,amssymb,amsthm,amscd,enumerate}
\usepackage{xcolor}
\usepackage{mathrsfs}
\usepackage[title]{appendix}
\usepackage{newtxtext,newtxmath}
\usepackage[a4paper, margin=1in]{geometry}

\AtBeginDocument{\fontsize{12pt}{14pt}\selectfont}
\numberwithin{equation}{section}
\newtheorem{theorem}{Theorem}[section]

\newtheorem{proposition}{Proposition}[section]
\newtheorem{lemma}{Lemma}[section]
\newtheorem{definition}{Definition}[section]
\newtheorem{corollary}{Corollary}[section]
\newtheorem{remark}{Remark}[section]

\title{Sharp Stability of Global Compactness on the Heisenberg Group}
\author{Hua Chen}
\author{Yun-lu Fan}
\author{Xin Liao}

\address{School of Mathematics and Statistics, Wuhan University, Wuhan 430072, China}

\date{\today}
\begin{document}
\begin{frontmatter}
\begin{abstract}
In this paper, we establish a sharp stability inequality on the Heisenberg group $\mathbb{H}^n$  for functions $u\in \mathscr{D}^1$ that are close to
the sum of $m$   weakly interacting  Jerison-Lee bubbles. Specifically, we show that the following  sharp estimates hold:
\begin{equation*}
\inf\limits_{\lambda_i ,\xi_i}\|u-\sum_{i=1}^{m}\mathfrak{g}_{\lambda_i ,\xi_i } U\|_{\mathscr{D}^{1}}\lesssim 
\begin{cases}
\|\Delta_{\mathbb{H}^n}u+|u|^{\frac{2}{n}}u \|_{\mathscr{D}^{-1}}, & \mbox{if } n=1 ,\\
\|\Delta_{\mathbb{H}^n}u+|u|^{\frac{2}{n}}u \|_{\mathscr{D}^{-1}}|\log  \|\Delta_{\mathbb{H}^n}u+|u|^{\frac{2}{n}}u \|_{\mathscr{D}^{-1}}|^{\frac{1}{2}}, & \mbox{if } n=2 ,\\
\|\Delta_{\mathbb{H}^n}u+|u|^{\frac{2}{n}}u \|_{\mathscr{D}^{-1}}^{\frac{n+2}{2n}}, & \mbox{if } n\geq 3,
\end{cases}
\end{equation*}
where the infimum is taken over all parameters $\{\lambda_i, \xi_i\}_{i=1}^m$ and $\inf\limits_{\lambda_i ,\xi_i}\|u-\sum_{i=1}^{m}\mathfrak{g}_{\lambda_i ,\xi_i } U\|_{\mathscr{D}^{1}}$ denotes the distance from $u$ to the manifold of sums of $m$ Jerison-Lee bubbles in the Folland-Stein-Sobolev space $\mathscr{D}^1$. As a consequence, we obtain a sharp  quantitative stability of global compactness for non-negative functions on Heisenberg group.
\end{abstract}
\end{frontmatter}
\section{Introduction}

$\mathbf{Notations}$:
Throughout the paper, the symbol $C$ will denote a positive constant, which may vary from line to line and may depend on  the number of bubbles and the space dimension. The term $o_{R}(1)$  represents an infinitesimal quantity with respect to $R$ as $R\to+\infty$. Furthermore, we define  $o_{R}(f)=o_{R}(1)f$.
We say  $g=O(f)$ if $|g|\leq C|f|$ for some constant $C$.
The notation
$X\lesssim Y$ means that there exists a positive constant $C$ such that $X\leq CY$. Similarly, $X\gtrsim Y$  means that there exists a positive constant
$C$ such that  $X\geq CY$.
We say $X\approx Y$  if both $X \lesssim Y$ and $X\gtrsim Y$ hold.
Additionally, $u^{+}$ and $u^{-}$ denote the positive and negative parts of $u$, respectively.
The symbol $Q:=2n+2$ denotes the homogeneous dimension of the Heisenberg group $\mathbb{H}^{n}$, and we set $p=\frac{Q+2}{Q-2}$.
\subsection{Motivation}
In recent years, considerable attention has been directed toward the quantitative stability of various functional inequalities and PDEs (cf. \cite{Bonforte, Ciraolo, FigalliNeumayer, Frank, FuscoMaggiPratelli, LiuZhangZou, Wei2022}).
To illustrate the concept, we begin with the Sobolev inequality:
 for $n\geq 3$ and $u\in \dot{H}^{1}(\mathbb{R}^n)$,
\[\bar{S}\|u\|_{L^{2^*}}^2\leq \|\nabla u\|_{L^2}^2,\]
where $2^*=\frac{2n}{n-2}$ and $\bar{S}>0$   is the optimal constant.
The extremal functions  are given by the Aubin–Talenti bubbles (cf. \cite{Aubin, Talenti}):
$$\bar{U}[\lambda, y](x)=(n(n-2))^{{\frac{n-2}{4}}}(\frac{\lambda}{1+\lambda^2|x-y|^2})^{\frac{n-2}{2}}.$$
These bubbles are precisely all the positive solutions to the classical Yamabe equation on $\mathbb{R}^n$:
\begin{equation}\label{yamabe}
-\Delta u=|u|^{2^*-2}u~~\mbox{in }  \mathbb{R}^n.
\end{equation}
Bianchi and Egnell \cite{Bianchi} established the quantitative stability of the Sobolev inequality: 
\begin{equation}\label{1.2}
 \inf\limits_{y,\lambda ,c}\|\nabla( u- c\bar{U}[\lambda, y])\|_{L^2} \leq C(n)  (\|\nabla u\|_{L^2}^2 -\bar{S}\|u\|_{L^{2^*}}^2),\quad \forall u\in \dot{H}^{1}(\mathbb{R}^n).
\end{equation}
 For further extensions, see \cite{FigalliZhang, LiuZhangZou} and the references therein.

A more delicate and challenging question is the stability of the Yamabe equation \eqref{yamabe} itself:
 whether a function  that almost solves \eqref{yamabe}  must be quantitatively close to
Aubin–Talenti bubbles?
Struwe's celebrated global compactness result (cf. \cite{Struwe}) addresses this question in a qualitative way:
 \begin{proposition}
  Let $\{u_k\} \subset \dot{H}^{1}(\mathbb{R}^n)$ be a sequence of nonnegative functions  with $$(m-\frac{1}{2})\bar{S}^{\frac{n}{2}}\leq \|u_k\|_{ \dot{H}^{1}}
  \leq (m+\frac{1}{2})\bar{S}^{\frac{n}{2}} $$  for some positive integer $m$, and
  \[ \|\Delta u_k+|u_k|^{2^*-2}u_k   \|_{\dot{H}^{-1}} \to 0 ~\mbox{as } k\to+\infty. \]
  Then there exists sequence of parameters $\{\lambda_i^{(k)}, y_i^{(k)}\}$ such that
  \[ \|u_k-\sum_{i=1}^{m} \bar{U}[\lambda_i^{(k)}, y_i^{(k)}] \|_{\dot{H}^{1}} \to 0 ~\mbox{as } k\to+\infty. \]
  Furthermore, for $i\neq j$, we have
  \[ \bar{\varepsilon}_{ij}^{(k)}:=\min \{ \frac{\lambda_i^{(k)}}{\lambda_j^{(k)}},  \frac{\lambda_j^{(k)}}{\lambda_i^{(k)}}, \frac{1}{\lambda_i^{(k)}\lambda_j^{(k)} |y_i^{(k)}-y_j^{(k)}|^2} \} \to 0 ~\mbox{as } k\to+\infty.   \]
\end{proposition}

Motivated by this, it is natural to investigate and compare the asymptotic rates of convergence of the above quantities as they tend to zero.
\begin{definition}
Let $\{\bar{U}[\lambda_i, y_i]\}_{i=1}^{m}$
 be a finite collection of Aubin–Talenti bubbles.
 For each pair $i\neq j$, define the interaction parameter
$$ \bar{\varepsilon}_{ij}:=\min \{ \frac{\lambda_i}{\lambda_j},  \frac{\lambda_j}{\lambda_i}, \frac{1}{\lambda_i\lambda_j |y_i-y_j|^2} \}.$$
 We say these bubbles are  $\delta$-weakly interacting if $$\bar{\varepsilon}:=\max_{i\neq j}\{\bar{\varepsilon}_{ij}\}\leq \delta.$$
 \end{definition}

Figalli and Glaudo \cite{Figalli2020} initiated the study of quantitative stability for the Yamabe equation near finite sums of weakly interacting Talenti bubbles, establishing sharp estimates in dimensions $3\leq n \leq 5$.
Later, Wei, Deng, and Sun \cite{Wei} extended these results by proving sharp quantitative stability estimates for $n\geq 6$.

We summarize their results (in \cite{Wei,Figalli2020}) as follows:

 \begin{theorem}
 Let $m\in \mathbb{N}$. There exists constant $\widetilde{\delta}=\widetilde{\delta}(n, m)$ such that for any $\delta \in (0, \widetilde{\delta})$, the following holds:
Assume $u\in \dot{H}^{1}(\mathbb{R}^n)$ satisfies
\[\|u-\sum_{i=1}^{m}\bar{U}[\widetilde{\lambda_i},\widetilde{y_i}] \|_{\dot{H}^{1}} <\delta,\]
for some collection of  $\delta$-weakly interaction  bubbles  $\{\bar{U}[\widetilde{\lambda_i},\widetilde{y_i}]\}_{i=1}^{m}$.
Then there exists another family of bubbles $\{\bar{U}[\lambda_i, y_i]\}_{i=1}^{m}$
  such that the following sharp stability estimate holds
\begin{equation*}
\|u-\sum_{i=1}^{m}\bar{U}[\lambda_i, y_i] \|_{\dot{H}^{1}} \lesssim
\begin{cases}
 \Gamma(u), & \mbox{if } 3\leq n\leq 5, \\
  \Gamma(u)|\log \Gamma(u) |^{\frac{1}{2}}, & \mbox{if } n=6, \\
  \Gamma(u)^{\frac{n+2}{2(n-2)}}, & \mbox{if } n\geq 7,
\end{cases}
\end{equation*}
where
$\Gamma(u):=\|\Delta u+|u|^{2^*-2}u\|_{\dot{H}^{-1}}.$ Moreover,   it holds that
 $\bar{\varepsilon}^{\frac{n-2}{2}} \lesssim \Gamma(u).$
\end{theorem}

\subsection{Our results}
A similar framework can be formulated on the Heisenberg group $\mathbb{H}^n\simeq \mathbb{C}^{n}\times\mathbb{R} \simeq \mathbb{R}^{2n+1}$.  We represent a point $\xi\in \mathbb{H}^n$ by
\[\xi:=(z,t)=(x+iy,t)\simeq (x,y,t)= (x_1,\ldots,x_n,y_1,\ldots, y_n, t)\in \mathbb{R}^{2n+1}. \]
For  $\xi,\xi'\in \mathbb{H}^n$, the group operation (left translation) is defined by
\[ \tau_{\xi'}(\xi):=\xi'\circ \xi := (z+z', t+t'+2\text{Im}(z'\cdot \bar{z} ) ). \]
The family of non-isotropic dilations $\{\delta_{\mu}\}_{\mu>0}$  is given by
\begin{equation*}
  \xi\longmapsto \delta_{\mu}(\xi):=(\mu x,\mu y, \mu^2t ).
\end{equation*}
The homogeneous dimension with respect to  $\{\delta_{\mu}\}_{\mu>0}$ is $Q=2n+2$.
A homogeneous norm on $\mathbb{H}^n$ is  defined by
\begin{equation}\label{eq:heisenberg-norm}
|\xi|= \left(|z|^4 + t^2\right)^{\frac{1}{4}},~ |z|^2=\sum_{j=1}^{n}(x_j^2+y_j^2).
\end{equation}
The triangle inequality associated with the norm $|\cdot|$ and left translation is discussed in Cygan \cite{Cygan1981}. The norm  \eqref{eq:heisenberg-norm} defines a metric $d$ given by $$d(\xi, \xi') =|{\xi'}^{-1}\circ \xi|. $$
 We denote by $B_{r}(\xi_0)$ the ball of radius
$r$ centered at  $\xi_0$ with respect to $d$, and write $B_{r}=B_{r}(0)$ when the center is the origin.
One can check that 
\begin{equation*}
  |B_r(\xi)|=r^Q|B_1|,~~~~~~r>0,
\end{equation*}
 since the Lebesgue measure $\, d\xi$ is a Haar measure on $\mathbb{H}^n$.

 The vector fields
\begin{equation*}
X_{j}=\frac{\partial}{\partial x_j}+2y_j\frac{\partial}{\partial t},~
X_{j+n}=\frac{\partial}{\partial y_j}-2x_j\frac{\partial}{\partial t},~ j=1,\ldots, n
\end{equation*}
generate the real Lie algebra of left-invariant vector fields on $\mathbb{H}^n$.
The Heisenberg Laplacian is defined as
\begin{equation*}
\Delta_{\mathbb{H}^n} :=-  \sum_{i=1}^{2n} X_i^*X_i
=\sum_{i=1}^{2n} X_i^2.
\end{equation*}
For function $u$, we set the horizontal gradient
\[Xu:= (X_1u,\ldots,X_nu, X_{n+1}u,\ldots,X_{2n} u),\]
and define the standard Folland-Stein-Sobolev space $\mathscr{D}^{1}$ as the completion of $C_0^{\infty}(\mathbb{H}^n)$ with respect to the norm $$\|u\|_{\mathscr{D}^{1}}:=\|Xu\|_{L^2}=(\sum_{i=1}^{2n}\|X_i u\|_{L^2}^2)^{\frac{1}{2}}.$$ The following Sobolev-type inequality  is well known to hold (cf. \cite{Jerison1988}):
\begin{equation*}\label{sob}
  S \|u\|^2_{L^{2^*_{Q}}}\leq \|X u\|^2_{L^2},~~~~\mbox{for any }u\in \mathscr{D}^{1},
\end{equation*}
 where        $2_Q^*=\frac{2Q}{Q-2}=\frac{2}{n}+2$ and $S$ is the optimal constant. The corresponding extremal functions are known as Jerison-Lee bubbles (see \cite{Jerison1988}):
       $$ \mathfrak{g}_{\lambda, \xi} U(\cdot):=  \lambda^{\frac{Q-2}{2}} U(\delta_{\lambda} ( \tau_{\xi^{-1}} (\cdot)) ),    $$
    where $$U(\xi)=\frac{c_0}{[(1+|z|^2)^2 +t^2]^{\frac{Q-2}{4}}} $$ and $c_0$ is a suitable constant such that
\begin{equation}\label{CRyamabe}
\Delta_{\mathbb{H}^n} U+|U|^{2^*_Q-2}U =\Delta_{\mathbb{H}^n} U+|U|^{\frac{2}{n}}U=0.
\end{equation}
Loiudice \cite{Loiudice}  had proved a version of \eqref{1.2} for $u\in \mathscr{D}^{1}$ on $\mathbb{H}^n$:
\[\inf\limits_{\xi,\lambda ,c}\|X( u- c\mathfrak{g}_{\lambda, \xi}U)\|_{L^2} \leq C(n)  (\|X u\|_{L^2}^2 -S\|u\|_{L^{2_Q^*}}^2). \]

We denote $\mathscr{D}^{-1}$  as the dual space of $\mathscr{D}^1$, with the dual pairing
$$(u,v)_{\mathscr{D}^1,\mathscr{D}^{-1}}:= \int_{\mathbb{R}^d} uv dx, ~~\forall u\in \mathscr{D}^1, v\in \mathscr{D}^{-1} .$$
Using the concentration–compactness principle (cf. Tintarev \cite{Tintarev2007}, Citti \cite{Citti2001}),   for  the CR Yamabe problem \eqref{CRyamabe}, we have the following global compactness result:
\begin{proposition}\label{prop1.2}
Let $\{u_k\} \subset \mathscr{D}^{1}$ be a sequence of nonnegative functions  such that
 $$(m-\frac{1}{2})S^{\frac{Q}{2}}\leq \|u_k\|_{ \mathscr{D}^{1}}
  \leq (m+\frac{1}{2})S^{\frac{Q}{2}} $$
    for some positive integer $m$, and
  \[ \|\Delta_{\mathbb{H}^n}u_k +|u_k |^{\frac{2}{n}}u_k   \|_{\mathscr{D}^{-1}} \to 0 ~\mbox{as } k\to+\infty. \]
  Then there exist parameters  $\{\lambda_i^{(k)}, \xi_i^{(k)}\}$ such that
  \[ \|u_k-\sum_{i=1}^{m} \mathfrak{g}_{\lambda_i^{(k)}, \xi_i^{(k)}}U\|_{\mathscr{D}^{1}} \to 0 ~\mbox{as } k\to+\infty. \]
  Moreover, for $i\neq j$,  we have
  \[ \varepsilon_{ij}^{(k)}:=\min \{ \frac{\lambda_i^{(k)}}{\lambda_j^{(k)}},  \frac{\lambda_j^{(k)}}{\lambda_i^{(k)}}, \frac{1}{\lambda_i^{(k)}\lambda_j^{(k)} |(\xi_i^{(k)})^{-1}\circ\xi_j^{(k)}|^2} \} \to 0 ~\mbox{as } k\to+\infty.   \]
\end{proposition}

Similarly, we can define the  interaction parameters in Heisenberg group settings.
\begin{definition}
Let $\{\mathfrak{g}_{\lambda_i, \xi_i}U\}_{i=1}^{m}$ be
 a finite collection of  Jerison-Lee bubbles.
For $i\neq j$,
 define
$$ \varepsilon_{ij}:=\min \{ \frac{\lambda_i}{\lambda_j},  \frac{\lambda_j}{\lambda_i}, \frac{1}{\lambda_i\lambda_j |\xi_i^{-1}\circ \xi_j|^2} \}.$$
 We say these bubbles are  $\delta$-weakly interacting if $$\varepsilon:=\max_{i\neq j}\{\varepsilon_{ij}\}\leq \delta.$$
 \end{definition}

Now, we present our main result as follows
\begin{theorem}\label{thm1-2}
For each $m\in \mathbb{N}$. There exists constant $\widetilde{\delta}=\widetilde{\delta}(n, m)$ such that for any $\delta \in (0, \widetilde{\delta})$, the following holds:
   if $u\in \mathscr{D}^{1}$ satisfies  \[\|u-\sum_{i=1}^{m}\mathfrak{g}_{\tilde{\lambda_i },\tilde{\xi_i }} U \|_{\mathscr{D}^{1}} <\delta,\]
   for a family of $\delta$-weakly interacting Jerison-Lee bubbles $\{\mathfrak{g}_{\tilde{\lambda_i} , \tilde{\xi_i }} U\}_{i=1}^{m}$ .
Then there exist  Jersion-Lee bubbles $\{\mathfrak{g}_{\lambda_i ,\xi_i } U\}_{i=1}^m $ such that
\begin{equation*}
\|u-\sum_{i=1}^{m}\mathfrak{g}_{\lambda_i ,\xi_i } U\|_{\mathscr{D}^{1}} \lesssim
\begin{cases}
\|\Delta_{\mathbb{H}^n}u+|u|^{\frac{2}{n}}u \|_{\mathscr{D}^{-1}}, & \mbox{if } n=1 ,\\
\|\Delta_{\mathbb{H}^n}u+|u|^{\frac{2}{n}}u \|_{\mathscr{D}^{-1}}|\log  \|\Delta_{\mathbb{H}^n}u+|u|^{\frac{2}{n}}u \|_{\mathscr{D}^{-1}}|^{\frac{1}{2}}, & \mbox{if } n=2 ,\\
\|\Delta_{\mathbb{H}^n}u+|u|^{\frac{2}{n}}u \|_{\mathscr{D}^{-1}}^{\frac{n+2}{2n}}, & \mbox{if } n\geq 3,
\end{cases}
\end{equation*}
where $\frac{n+2}{2n}=\frac{Q+2}{2(Q-2)}$. Furthermore, 
$$\varepsilon^{n} \lesssim \|\Delta_{\mathbb{H}^n}u+|u|^{\frac{2}{n}}u \|_{\mathscr{D}^{-1}}.$$
\end{theorem}
\begin{remark}
In contrast to \cite{Wei},  instead of relying on intricate pointwise estimates that may fail on the Heisenberg group,  we obtain a more natural energy estimate in Lemma  \ref{main2} by making full use of the concentration-compactness method and the nondegeneracy of the solution. In addition
we simplify the estimations by making use of the symmetry properties discussed in Section 2.2. 
\end{remark}
As a direct consequence of Theorem \ref{thm1-2} and Proposition \ref{prop1.2}, we have the following quantitative stability of global compactness for non-negative functions on $\mathbb{H}^n$.
\begin{corollary}\label{cor1-1}
Let $u \in \mathscr{D}^{1}$ be a sequence of nonnegative functions such that $$(m-\frac{1}{2})S^{\frac{Q}{2}}\leq \|u\|_{ \mathscr{D}^{1}}
  \leq (m+\frac{1}{2})S^{\frac{Q}{2}}. $$
  Then there exist  Jersion-Lee bubbles $\{\mathfrak{g}_{\lambda_i , \xi_i } U\}_{i=1}^{m}$ such that
 \begin{equation*}
\|u-\sum_{i=1}^{m}\mathfrak{g}_{\lambda_i ,\xi_i } U\|_{\mathscr{D}^{1}} \lesssim
\begin{cases}
\|\Delta_{\mathbb{H}^n}u+|u|^{\frac{2}{n}}u \|_{\mathscr{D}^{-1}}, & \mbox{if } n=1 ,\\
\|\Delta_{\mathbb{H}^n}u+|u|^{\frac{2}{n}}u \|_{\mathscr{D}^{-1}}|\log  \|\Delta_{\mathbb{H}^n}u+|u|^{\frac{2}{n}}u \|_{\mathscr{D}^{-1}}|^{\frac{1}{2}}, & \mbox{if } n=2 ,\\
\|\Delta_{\mathbb{H}^n}u+|u|^{\frac{2}{n}}u \|_{\mathscr{D}^{-1}}^{\frac{n+2}{2n}}, & \mbox{if } n\geq 3.
\end{cases}
\end{equation*}
\end{corollary}

 Finally,  we construct examples to demonstrate the sharpness of Theorem \ref{thm1-2} and Corollary \ref{cor1-1} in the following sense:
 \begin{theorem}\label{thmsharp}
 For any $\delta>0$, there exist $\varepsilon\in (0, \delta)$
and a  nonnegative function $u$ such that
$$\|u(z,t)-U(z,t)-U(z,t+\frac{1}{\varepsilon})\|_{\mathscr{D}^{1}}\leq \delta,$$ but
  \begin{equation*}
 \inf_{\lambda_i, \xi_i}\| u-\sum_{i=1}^{2}\mathfrak{g}_{\lambda_i, \xi_i }U \|_{\mathscr{D}^1}\gtrsim
\begin{cases}
\|\Delta_{\mathbb{H}^n}u+|u|^{\frac{2}{n}}u \|_{\mathscr{D}^{-1}}|\log  \|\Delta_{\mathbb{H}^n}u+|u|^{\frac{2}{n}}u \|_{\mathscr{D}^{-1}}|^{\frac{1}{2}}, & \mbox{if } n=2 ,\\
\|\Delta_{\mathbb{H}^n}u+|u|^{\frac{2}{n}}u \|_{\mathscr{D}^{-1}}^{\frac{n+2}{2n}}, & \mbox{if } n\geq 3.
\end{cases}
\end{equation*}
\end{theorem}
\begin{remark}
  When $n=1$, the estimate in Theorem \ref{thm1-2} is automatically sharp, which can be seen by considering the simple perturbation \( u_{\epsilon} = (1 + \epsilon)Q \) for some small \(\epsilon>0 \).
\end{remark}

\section{Preliminaries}
\subsection{Basic of  Jersion-Lee bubble}
In this subsection, we introduce some basic properties of the Jersion–Lee bubble.
Recall the standard bubble
  $$U(z,t)= \frac{c_0}{[(1+|z|^2)^2 +t^2]^{\frac{Q-2}{4}}},$$
and its transformed version under scaling and translations:
  $$\mathfrak{g}_{\lambda, \xi_0}U(z,t)=  \frac{c_0\lambda^{\frac{Q-2}{2}}}{[(1+\lambda^2|z- z_0|^2)^2 +\lambda^4(t-t_0-2 \text{Im} (\overline{z}\cdot z_0) )^2]^{\frac{Q-2}{4}}} .$$
  For $\eta \in \mathbb{R}^{2n+1}\simeq\mathbb{H}^n$ , let  $\eta^{(a)}$ denote its $a$-th coordinate.
we define
  \[Z^{a}=\frac{\partial \mathfrak{g}_{1, \eta}U }{\partial\eta^{(a)}}|_{ \eta=0},~ a=1,\ldots, 2n+1.\]
  In addition, we define
\begin{equation}\label{Z2n}
\begin{aligned}
Z^{2n+2}(\xi)&= \frac{\partial \mathfrak{g}_{\lambda, 0}U }{\partial\lambda}|_{\lambda=1}(\xi)=\frac{Q-2}{2}U(\xi)+TU(\xi) \\& = \frac{Q-2}{2}U(\xi)-(Q-2)\frac{|z|^2(1+|z|^2)+t^2}{(1+|z|^2)^2+t^2} U(\xi),
\end{aligned}
\end{equation}
where the vector field $T$ is given by $$T=\sum_{j=1}^{n} (x_{j}\frac{\partial}{\partial x_j} + y_{j}\frac{\partial}{\partial y_j})+2t\frac{\partial}{\partial t}.$$
As a consequence of integrating by parts, since $\mbox{div }T=Q$,  we obtain:
  \begin{equation}\label{Z}
  p\int_{\mathbb{H}^n } TU \cdot U^{p-1} \, d\xi=\int_{\mathbb{H}^n } T(U^p) \, d\xi= -Q \int_{\mathbb{H}^n }  U^p \, d\xi.
  \end{equation}
  Furthermore, the following estimate holds:
\begin{equation*}|Z^{a}|\lesssim U,~\forall a=1,\ldots, 2n+2.
\end{equation*}

In order to establish our results, we require the following non-degeneracy property of Jersion-Lee bubble ( recall that $p=\frac{Q+2}{Q-2}$):
\begin{lemma}\label{nonde}
  If $u \in \mathscr{D}^1$ is a solution to $$-\Delta_{
  \mathbb{H}^n}u=pU^{p-1}u ~\text{in}~ \mathbb{H}^n,$$
  then $u$ must be a linear combination of the functions  $\{Z^{a}\}_{a=1}^{2n +2}$.
\end{lemma}
\begin{proof}
  See Lemma 5 in Malchiodi-Uguzzoni \cite{Malchiodi}.
\end{proof}

\subsection{Interaction between bubbles}
We define the set of all gauge transformations associated with the standard bubble as
\begin{equation*}\label{gauge}
\mathscr{G}:=\{\mathfrak{g}_{\lambda, \xi}\mid \lambda >0, \xi \in \mathbb{H}^n\}.
\end{equation*}
It is straightforward to verify that each \(\mathfrak{g} \in \mathscr{G}\) is an isometry on both \(\mathscr{D}^{1}\) and \(L^{2_Q^*}\).
Moreover,
\(\mathscr{G}\) forms a group under composition.
For the sake of convenience, we will occasionally denote
 $\mathfrak{g}_{\lambda_i, \xi_i}$ by $\mathfrak{g}_i$. By direct calculation, we obtain the following relationship:
\begin{equation}\label{sym5}
\mathfrak{g}_i^{-1}\mathfrak{g}_{j}=\mathfrak{g}_{\frac{\lambda_j}{\lambda_i}, \delta_{\lambda_i}(\xi_i^{-1}\circ \xi_j)}.
\end{equation}

The following lemma holds:
\begin{lemma}\label{gg}
For a given sequence $\{\lambda_k,\xi_k \}_{k=1}^{\infty}$, the following statements are equivalent:
  \begin{itemize}
    \item For each $u\in \mathscr{D}^{1}$, $\mathfrak{g}_k u\rightharpoonup 0$ weakly in $\mathscr{D}^{1}$.
    \item For each $u\in \mathscr{D}^{1}$, $\mathfrak{g}_k^{-1} u\rightharpoonup 0$ weakly in $\mathscr{D}^{1}$.
    \item $|\log \lambda_k| +|\xi_k| \to +\infty.$
    \item $|\log \lambda_k| +\lambda_k |\xi_k| \to +\infty.$
  \end{itemize}
   In any of these cases, we say $\mathfrak{g}_k \rightharpoonup 0$.
\end{lemma}
\begin{proof}
See \cite[Lemma 9.13]{Tintarev2007}.
\end{proof}
\begin{remark}
As a consequence of Lemma \ref{gg}, for sequences $\{\mathfrak{g}_{\lambda_1^{(k)}, \xi_1^{(k)}}\}_{k=1}^{\infty}$ and $\{\mathfrak{g}_{\lambda_2^{(k)}, \xi_2^{(k)}}\}_{k=1}^{\infty}$ we obtain
that $(\mathfrak{g}_{\lambda_1^{(k)}, \xi_1^{(k)}})^{-1}\mathfrak{g}_{\lambda_2^{(k)}, \xi_2^{(k)}} \rightharpoonup 0$ if and only if $\varepsilon_{12}^{(k)}\to 0$ , i.e.,
\[ \min \{ \frac{\lambda_1^{(k)}}{\lambda_2^{(k)}},  \frac{\lambda_2^{(k)}}{\lambda_1^{(k)}}, \frac{1}{\lambda_1^{(k)}\lambda_2^{(k)} |(\xi_1^{(k)})^{-1}\circ\xi_2^{(k)}|^2 }\} \to 0 .\]
 
\end{remark}

For each $g \in \mathscr{G}$, the following identity holds:
\begin{equation}\label{sym1}
(\mathfrak{g}u, v)_{\mathscr{D}^{1}} =(u, \mathfrak{g}^{-1}v)_{\mathscr{D}^{1}}, \forall u, v\in \mathscr{D}^{1}.
\end{equation}
For any $\alpha, ~\beta,~ \gamma \geq 0$ with $\alpha+\beta+\gamma=2_Q^*$, and $u,~ v,~w \in L^{2_Q^*}$, we have the invariance property
\begin{equation}\label{sym2}
\int_{ \mathbb{H}^n} |u|^{\alpha} |v|^{\beta} |w|^{\gamma}  \, d\xi =\int_{ \mathbb{H}^n} |\mathfrak{g}u|^{\alpha} |\mathfrak{g}v|^{\beta} |\mathfrak{g}w|^{\gamma}  \, d\xi,
\end{equation}
and if  there is also $ \alpha>0$, then
\begin{equation}\label{sym3}
\int_{ \mathbb{H}^n} |(u+v)^{\alpha}-u^{\alpha}-v^{\alpha} |^{\frac{2_Q^*}{\alpha}}  \, d\xi =\int_{ \mathbb{H}^n} |(\mathfrak{g}u+\mathfrak{g}v)^{\alpha}-(\mathfrak{g}u)^{\alpha}-(\mathfrak{g}v)^{\alpha} |^{\frac{2_Q^*}{\alpha}}   \, d\xi.
\end{equation}
Additionally,
for $u\in \mathscr{D}^{1}$ and $v\in \mathscr{D}^{-1}$, and $\mathfrak{g}=\mathfrak{g}_{\lambda, \eta}$, the following holds:
\begin{equation}
\int_{ \mathbb{H}^n} \mathfrak{g}u  v \, d\xi =\lambda^{-2} \int_{ \mathbb{H}^n}  u   \mathfrak{g}^{-1}v \, d\xi.
\end{equation}
Therefore, we obtain
\begin{equation}\label{sym4}
\|v\|_{\mathscr{D}^{-1}}=\|\lambda^{-2}\mathfrak{g} ^{-1}v \|_{\mathscr{D}^{-1}}.
\end{equation}

Giving $\{\lambda_i, \xi_i\}_{i=1}^{m}$, we will henceforth use the notations: $$\mathfrak{g}_i=\mathfrak{g}_{\lambda_i, \xi_i},~ U_i=\mathfrak{g}_iU,~\text{and}~  Z_i^a=\mathfrak{g}_i Z^a.$$
Suppose that 
$\varepsilon=\max\limits_{i\neq j}\varepsilon_{ij}$
is sufficiently small. Then by using equations  \eqref{Z2n}-\eqref{sym4} and Lemma \ref{app1}-\ref{app4} in Appendix, we conclude that:
\begin{enumerate}[$\bullet$]
  \item  For all $i\neq j, ~\alpha+\beta=2_Q^*$ and $\alpha\neq \beta\geq 0$, we have
\begin{equation}\label{e1}
 \int_{\mathbb{H}^n}|U_i|^{\alpha}|U_j|^{\beta} \, \, d\xi  \approx  \varepsilon_{ij}^{\frac{\min\{\alpha,\beta\}(Q-2)}{2}}.
\end{equation}
  \item For all $ i \neq j$ with $\lambda_i\leq \lambda_j$, we have
\begin{equation}\label{e2}
\begin{aligned}
\int_{\mathbb{H}^n} U_j^{p-1} U_i Z_j^{2n+2} \, d\xi 
&= \left( \frac{Q-2}{2} + \frac{\int_{\mathbb{H}^n} TU \cdot U^{p-1} \, d\xi }{\int_{\mathbb{H}^n} U^p \, d\xi } \right) \int_{\mathbb{H}^n} U_j^{p} U_i \, d\xi + o(\varepsilon_{ij}^{\frac{(Q-2)}{2}}) \\
&= -\frac{(Q-2)^2}{2(Q+2)} \int_{\mathbb{H}^n} U_j^{p} U_i \, d\xi + o(\varepsilon_{ij}^{\frac{(Q-2)}{2}}).
\end{aligned}
\end{equation}
  \item For all $i\neq j$ with $n>2$, we have:
\begin{equation}\label{e3}
\int_{\mathbb{H}^n}   |(U_i+U_j)^{p}-U_i^{p}-U_j^{p}|^{\frac{2Q}{Q+2}} \, d\xi\lesssim  \varepsilon_{ij}^{\frac{Q}{2}}.
\end{equation}
  \item For $n=2$, we have
\begin{equation}\label{e4}
\|U_iU_j\|_{\mathscr{D}^{-1}}\lesssim \varepsilon_{ij}^2|\log \varepsilon_{ij}|^{\frac{1}{2}}.
\end{equation}
\end{enumerate}

Furthermore, we have the following result.
\begin{lemma}\label{fh-1}
The function
$
f = ( \sum_{i=1}^m U_i )^p - \sum_{i=1}^m U_i^p.
$
 satisfies 
\begin{equation}\label{2.14}
\| f \|_{\mathscr{D}^{-1}} \lesssim 
\begin{cases}
\varepsilon^{\frac{Q+2}{4}}, & \text{if } n\geq 3, \\
\varepsilon^2 |\log \varepsilon|^{\frac{1}{2}}, & \text{if } n=2, \\
\varepsilon, & \text{if } n=1.
\end{cases}
\end{equation}
\end{lemma}
\begin{proof}
  For $n\geq 3$, we can deduce \eqref{2.14} from  
  \begin{equation}\label{norm-1}
  \|\cdot\|_{\mathscr{D}^{-1}}\lesssim \|\cdot\|_{L^{\frac{2Q}{Q+2}}},
  \end{equation} \eqref{e3} and the fact that (cf. \cite[Lemma A.6]{Wei})
  \begin{equation*}\label{qqq}
|(\sum_{i=1}^{m} U_i)^p-\sum_{i=1}^{m} U_i^p|\lesssim  \sum_{1\leq i\neq j\leq m} \Big{(} (U_i+U_j)^p-U_i^p-U_j^p \Big{)},   \mbox{    if } 1<p\leq 2.
\end{equation*}
For $n=2$, it follows from \eqref{e4} directly.
For $n=1$, since $p=\frac{n+2}{n}=3$, it follows from  \eqref{norm-1}, \eqref{e1} and H\"{o}lder inequality.
\end{proof}

Finally, we point out here that  Lemma \ref{fh-1} is sharp for $n\geq 2$.
To illustrate this, set $$V(z,t)=U(z,t+\varepsilon^{-1}),  ~~f=(U+V)^p-U^p-V^p.$$ 
The following lemma holds:
\begin{lemma}\label{sharpes}
Following the above notations,  for $\varepsilon\in (0,1)$, we have
\begin{equation*}\|f\|_{\mathscr{D}^{-1}}\gtrsim
\begin{cases}
  \varepsilon^{\frac{Q+2}{4}}, & \mbox{if } n\geq 3, \\
 \varepsilon^2|\log  \varepsilon|^{\frac{1}{2}}, & \mbox{if } n=2.
\end{cases}
\end{equation*}
\end{lemma}
\begin{proof}
For $n\geq3$, let  $\eta$ be a cutoff function such that $\eta=1$ in $B_{\frac{1}{4\sqrt{\varepsilon}}}$ and $\eta=0$ outside $B_{\frac{1}{2\sqrt{\varepsilon}}}$. 
Note that in $B_{\frac{1}{2\sqrt{\varepsilon}}}$, we have $$U\gtrsim V\approx \varepsilon^{\frac{Q-2}{2}} \mbox{ and } f\approx  U^{p-1}V.$$
Therefore, 
\begin{equation*}
\begin{aligned}
\varepsilon^{\frac{2-Q}{4}} \|f\|_{\mathscr{D}^{-1}} &\gtrsim \|\eta\|_{\mathscr{D}^{1}}\|f\|_{\mathscr{D}^{-1}}\\&
\gtrsim
\int_{\mathbb{H}^n } f\eta \, d\xi
\\& \gtrsim \int_{B_{\frac{1}{2\sqrt{\varepsilon}}}} U^{p-1}V \, d\xi\\
&\gtrsim \varepsilon^{\frac{Q-2}{2}}\int_{B_{\frac{1}{2\sqrt{\varepsilon}}}} \frac{1}{1+|\xi|^4} \, d\xi\\
&\gtrsim \varepsilon.
\end{aligned}
\end{equation*}

For $n=2$,
we have $Q=6, p=2$ and
\begin{equation*}
\|f\|_{\mathscr{D}^{-1}}\|f^{\frac{1}{2}}\|_{\mathscr{D}^{1}}\gtrsim
\int_{\mathbb{H}^n } U^{\frac{3}{2}}V^{\frac{3}{2}} \, d\xi \gtrsim
\varepsilon^{3}\int_{B_{\frac{1}{2\sqrt{\varepsilon}}}} \frac{\, d\xi}{1+|\xi|^6}
\gtrsim  \varepsilon^{3} |\log \varepsilon|.
\end{equation*}
Directly calculation yields
\begin{align*}
    |X f^{\frac{1}{2}}|^2 \lesssim |XU|^2U^{-1}V+   |XV|^2V^{-1}U , \\
   |XU(\xi)|^2\approx \frac{|z|^2}{((1+|z|^2)^2 +t^2 )^3} \lesssim U(\xi)^{\frac{5}{2}}.
\end{align*}
By setting $\xi_0=(0, -\varepsilon^{-1})$, we thus obtain
\begin{equation*}
\begin{aligned}
\|f^{\frac{1}{2}}\|_{\mathscr{D}^{1}}^2
&\lesssim\int_{\mathbb{H}^n } U^{\frac{3}{2}}V \, d\xi\\&
\lesssim \varepsilon^{2}\int_{B_{\frac{1}{2\sqrt{\varepsilon}}}} \frac{\, d\xi}{1+|\xi|^6}
+ \varepsilon^{3} \int_{ B_{\frac{1}{2\sqrt{\varepsilon}}}^c \cap B_{\frac{1}{2\sqrt{\varepsilon}}}(\xi_0)}  \frac{\, d\xi}{1+|\xi_0^{-1}\circ \xi|^4}
+\int_{B_{\frac{1}{2\sqrt{\varepsilon}}}^c\setminus B_{\frac{1}{2\sqrt{\varepsilon}}}(\xi_0)} \frac{\, d\xi}{1+|\xi|^{10}}
\\&
\lesssim \varepsilon^{2}|\log \varepsilon|.
\end{aligned}
\end{equation*}
\end{proof}

\section{Proof of  Theorem \ref{thm1-2}}
\subsection{Sketch of the proof}
Under the assumption of Theorem~\ref{thm1-2}, and following an argument similar to Appendix A in Bahri--Coron~\cite{Bahri},  the smooth function
\[
G(\lambda_1, \ldots, \lambda_m, \xi_1, \ldots, \xi_m) = \| u - \sum_{k=1}^{m} \mathfrak{g}_{\lambda_k, \xi_k} U \|_{\mathscr{D}^{1}}^2,
\]
 attains its minimum at some point \((\lambda_1, \ldots, \lambda_m, \xi_1, \ldots, \xi_m)\) satisfying
\[
\frac{\lambda_i}{\widetilde{\lambda}_i} = 1 + o_{\delta}(1), 
\quad 
\lambda_i \widetilde{\lambda}_i |\xi_i^{-1} \circ \widetilde{\xi}_i |^2 = o_{\delta}(1),
\quad \forall\, 1\leq i\leq m.
\]
Set
\[
U_i = \mathfrak{g}_{\lambda_i, \xi_i} U, 
\quad 
\sigma = \sum_{i=1}^{m} U_i, \quad Z_i^a = \mathfrak{g}_{\lambda_i, \xi_i} Z^a, 
\quad
\rho = u - \sigma.
\]
By differentiating $G$, we derive the following orthogonality condition:
\begin{equation*}
\left( \rho, \frac{\partial \mathfrak{g}_{\lambda_i, \eta}U}{\partial \eta^{(a)}}\Big|_{\eta=\xi_i} \right)_{\mathscr{D}^{1}} = 0, 
\quad 
\left( \rho, \frac{\partial \mathfrak{g}_{r, \xi_i}U}{\partial r}\Big|_{r=\lambda_i} \right)_{\mathscr{D}^{1}} = 0,
\quad \forall\, 1\leq i\leq m,\ 1\leq a\leq 2n+1.
\end{equation*}
It can be verified that the sets
\[
\left\{ \frac{\partial \mathfrak{g}_{\lambda_i, \eta}U}{\partial \eta^{(a)}}\Big|_{\eta=\xi_i} \right\}_{1\leq a\leq 2n+1}
\cup 
\left\{ \frac{\partial \mathfrak{g}_{r, \xi_i}U}{\partial r}\Big|_{r=\lambda_i} \right\}
\quad \text{and} \quad
\left\{ Z_i^a \right\}_{1\leq a\leq 2n+2}
\]
span the same subspace for each \( 1\leq i\leq m \), and $$-\Delta_{\mathbb{H}^n }  Z_{i}^{a}=pU_{i}^{p-1}Z_{i}^{a}, \quad \forall\, 1\leq i\leq m,\ 1\leq a\leq 2n+1.$$
Consequently, we obtain
\begin{equation}\label{orthogonality}
\int_{\mathbb{H}^n}  pU_{i}^{p-1} \rho Z_i^a  \, d\xi   =\left( \rho, Z_i^a \right)_{\mathscr{D}^{1}} = 0, \quad \forall\, 1\leq i\leq m,\ 1\leq a\leq 2n+2.
\end{equation}
\begin{proof}[Proof of the Theorem \ref{thm1-2}]
 By direct calculation, we obtain
\begin{equation}\label{key}
\begin{aligned}
-\Delta_{\mathbb{H}^n } \rho-p\sigma^{p-1} \rho
&= f+h +  N(\rho),
\end{aligned}
\end{equation}
where $$f:=(\sum\limits_{i=1}^{m}  U_i )^p -\sum\limits_{i=1}^{m}  U_i^p,~ h:=-\Delta_{\mathbb{H}^n } u-|u|^{p-1}u,$$ and
  \begin{align*}
  N(\rho):=|\sigma+\rho|^{p-1}(\sigma+\rho) -\sigma^p -p\sigma^{p-1} \rho.
  \end{align*}
  Proposition \ref{main1} yields that
    \begin{equation}\label{maineq1}
     \varepsilon^{\frac{Q-2}{2}} \lesssim  \|\rho\|_{\mathscr{D}^{1}}^2+\|h\|_{\mathscr{D}^{-1}}.
   \end{equation}
   By Lemma \ref{main2}, we obtain
   \begin{equation*}
   \begin{aligned}
   \|\rho\|_{\mathscr{D}^{1}}&\lesssim \|f+h+N(\rho)\|_{\mathscr{D}^{-1}}\\&
   \lesssim \|f\|_{\mathscr{D}^{-1}}+\|h\|_{\mathscr{D}^{-1}}+\|N(\rho)\|_{L^{\frac{2Q}{Q+2}}}.
   \end{aligned}
   \end{equation*}
   Since
   \begin{equation*}
   |N(\rho)| \lesssim
    \begin{cases}
    \sigma^{p-2}\rho^2+|\rho|^p, & \mbox{if } p>2 ,\\
    |\rho|^p, & \mbox{if } 1<p\leq2,
  \end{cases}
  \end{equation*}
  it follows from Lemma \ref{fh-1} that
  \begin{equation}\label{maineq2}
 \|\rho\|_{\mathscr{D}^{1}}  \lesssim \|f\|_{\mathscr{D}^{-1}}+\|h\|_{\mathscr{D}^{-1}} \lesssim  \begin{cases}
             \varepsilon^{\frac{Q+2}{4}} +\|h\|_{\mathscr{D}^{-1}}, & \mbox{if } n\geq 3, \\
             \varepsilon^2|\log \varepsilon|^{\frac{1}{2}}+\|h\|_{\mathscr{D}^{-1}}, & \mbox{if } n=2,\\
              \varepsilon+\|h\|_{\mathscr{D}^{-1}} , & \mbox{if } n=1.
           \end{cases}
  \end{equation}
  Thus, by combining \eqref{maineq1} and \eqref{maineq2}, we deduce that
  \begin{equation}
   \varepsilon^{\frac{Q-2}{2}} \lesssim \|h\|_{\mathscr{D}^{-1}},
  \end{equation}
  and
  \begin{equation}
   \|\rho\|_{\mathscr{D}^{1}}  \lesssim\begin{cases}
             \|h\|_{\mathscr{D}^{-1}}^{\frac{n+2}{2n}}, & \mbox{if } n\geq 3,\\
             \|h\|_{\mathscr{D}^{-1}}|\log \|h\|_{\mathscr{D}^{-1}}|^{\frac{1}{2}}, & \mbox{if } n=2,\\
             \|h\|_{\mathscr{D}^{-1}}, & \mbox{if } n=1.
           \end{cases}
  \end{equation}
\end{proof}
\begin{proof}[Proof of Corollary \ref{cor1-1}]
 This proof is standard. Suppose, on the contrary, that for each $k$, there exists non-negative function $u_k$, such that $$(m-\frac{1}{2})S^{\frac{Q}{2}}\leq \|u_k\|_{\mathscr{D}^1} \leq (m+\frac{1}{2})S^{\frac{Q}{2}},$$ and
\begin{equation}\label{mddd}
  \inf_{ \lambda_i ,\xi_i} \|u_k-\sum_{i=1}^{m}\mathfrak{g}_{\lambda_i ,\xi_i }U \|_{\mathscr{D}^1} > k\|\Delta_{\mathbb{H}^n } u_k+|u_k|^{\frac{2}{n}}u_k\|_{\mathscr{D}^{-1}}.
\end{equation}
 Thus, $\{u_k\}$ is a Palais–Smale sequence for the functional associated with  the Euler-Lagrange  equation $\Delta_{\mathbb{H}^n } u+|u|^{\frac{2}{n}}u$. Proposition \ref{prop1.2} yields that, up to a subsequence,
 there exist    $\delta_{k}$-weakly interacting Jerison-Lee bubbles  $\{ \mathfrak{g}_{\lambda_i^{(k)},\xi_i^{(k)} } U\}$ such that $$\delta_{k} \mbox{ and }\| u_k -\sum_{i=1}^m \mathfrak{g}_{\lambda_i^{(k)},\xi_i^{(k)} } U \|_{\mathscr{D}^1 } \to 0,~~\text{as}~k\to +\infty.$$
 By applying Theorem \ref{thm1-2}, we arrive at a contradiction.
\end{proof}

The detailed proofs of the main missing estimates will be provided in the next two subsections.
\subsection{The first missing estimate}
Before giving the proof of Proposition \ref{main1}, we need the following counterpart of Lemma 2.1 in \cite{Wei}:
\begin{lemma}\label{fz}
For fixed $k$, under the notations in Section 3.1, we have
\begin{equation}
\int_{\mathbb{H}^n } f Z_k^{2n+2} \, d\xi=p \sum_{i\neq k}\int U_k^{p-1}U_i Z_k^{2n+2}  \, d\xi
+o(\varepsilon^{\frac{Q-2}{2}}).
\end{equation}
\end{lemma}
\begin{proof}
  Denote $\sigma_{[k]}=\sum_{j\neq k }U_j$. Let's decompose $$\int_{\mathbb{H}^n } f Z_k^{2n+2}  \, d\xi=J_1+J_2+J_3+J_4,$$
  where
  \[ J_1=\int_{U_k>\sigma_{[k]}}[pU_k^{p-1}\sigma_{[k]}-\sum_{i\neq k}U_i^p] Z_k^{2n+2}  \, d\xi ,  \]
  \[J_2=\int_{U_k>\sigma_{[k]}}  [ \sigma^{p}-pU_k^{p-1}\sigma_{[k]}-U_k^p ]Z_k^{2n+2}  \, d\xi,\]
  \[ J_3=\int_{U_k\leq\sigma_{[k]}}  [ p\sigma_{[k]}^{p-1}U_k+\sigma_{[k]}^{p}-\sum_{i}U_i^p ]Z_k^{2n+2}  \, d\xi,\]
  \[J_4= \int_{U_k\leq\sigma_{[k]}}  [ \sigma^{p}-\sigma_{[k]}^{p}-p\sigma_{[k]}^{p-1}U_k ]Z_k^{2n+2}   \, d\xi.\]
By \eqref{e1} and  $|Z_{k}^{2n+2}|\lesssim U_k$, for some $\theta>0$ small enough,  we have
\begin{equation}\label{est-J1}
\begin{aligned}
\left| J_1 - p \int_{\mathbb{H}^n} U_k^{p-1} \sigma_{[k]} Z_k^{2n+2} \, d\xi \right|
&= \left| \int_{U_k > \sigma_{[k]}} \sum_{i\neq k} U_i^p Z_k^{2n+2} \, d\xi+ \int_{U_k \leq \sigma_{[k]}} p U_k^{p-1} \sigma_{[k]} Z_k^{2n+2} \, d\xi \right| \\
&\lesssim \sum_{i\neq k} \int_{\mathbb{H}^n} \left( U_i^{p-\theta} U_k^{1+\theta} + U_i^{1+\theta} U_k^{p-\theta}  \right)\, d\xi \\
&\lesssim o( \varepsilon^{\frac{Q-2}{2}}) .
\end{aligned}
\end{equation}

Next, for $J_2$, it holds that
\begin{equation}\label{est-J2}
\begin{aligned}
|J_2| 
&\lesssim \int_{U_k > \sigma_{[k]}} U_k^{p-1} \sigma_{[k]}^2  \, d\xi 
\lesssim \int_{\mathbb{H}^n} U_k^{p-\theta} \sigma_{[k]}^{1+\theta}  \, d\xi
\lesssim o( \varepsilon^{\frac{Q-2}{2}}) .
\end{aligned}
\end{equation}

Similarly, for $J_3$, we estimate
\begin{equation}\label{est-J3}
\begin{aligned}
|J_3| 
&\lesssim \int_{\mathbb{H}^n} \sigma_{[k]}^{p-\theta} U_k^{1+\theta}  \, d\xi
+ \int_{U_k \leq \sigma_{[k]}} \left( \sigma_{[k]}^p - \sum_{i\neq k} U_i^p \right) Z_k^{2n+2}  \, d\xi\\
&\lesssim o( \varepsilon^{\frac{Q-2}{2}}) + \sum_{\substack{i,j \\ i,j \neq k,\, i\neq j}} \int_{\mathbb{H}^n} U_i^{p-1} U_j U_k \, d\xi \\
&\lesssim o( \varepsilon^{\frac{Q-2}{2}}) ,
\end{aligned}
\end{equation}
where in the last inequality, we have used  \eqref{e1} and H\"{o}lder inequality.
Finally, for $J_4$, we obtain
\begin{equation}\label{est-J4}
\begin{aligned}
|J_4| 
&\lesssim \int_{U_k \leq \sigma_k} U_k^3 \sigma_{[k]}^{p-2}  \, d\xi
\lesssim \int_{\mathbb{H}^n} U_k^{1+\theta} \sigma_{[k]}^{p-\theta}  \, d\xi
\lesssim o( \varepsilon^{\frac{Q-2}{2}}) .
\end{aligned}
\end{equation}

Combining \eqref{est-J1}–\eqref{est-J4}, we complete the proof.
\end{proof}

 Consequently, we achieve that
 \begin{proposition}\label{main1}
 Under the assumption of Theorem \ref{thm1-2} and the notations in Section 3.1,
the following estimates hold:
   \begin{equation}
     \varepsilon^{\frac{Q-2}{2}} \lesssim  \|\rho\|_{\mathscr{D}^{1}}^2+\|h\|_{\mathscr{D}^{-1}}.
   \end{equation}
 \end{proposition}
 \begin{proof}
Testing equation \eqref{key} with $Z_{i}^{2n+2}$, we obtain
\begin{equation}\label{3-13}
\left| \int_{\mathbb{H}^n} f Z_i^{2n+2}  \, d\xi\right| 
\lesssim 
\left| \int_{\mathbb{H}^n} \sigma^{p-1} \rho Z_i^{2n+2}  \, d\xi \right| 
+ \|h\|_{\mathscr{D}^{-1}} 
+ \left|\int_{\mathbb{H}^n} N(\rho) Z_i^{2n+2} \, d\xi \right| .
\end{equation}

We first estimate the first term on the right-hand side of \eqref{3-13}.  
For some sufficiently small $\theta>0$, we have
\begin{equation}\label{4-111}
\begin{aligned}
| (\sigma^{p-1} - U_i^{p-1}) \rho Z_i^{2n+2} |
& \lesssim U_i | \sigma^{p-1} - U_i^{p-1} | |\rho| \\
& \lesssim
\begin{cases}
U_i^{p-1} |\rho|^2, & \text{if } U_i > \sigma_{[i]}, \, |\rho| > \sigma_{[i]}, \\
U_i^{p-\theta} \sigma_{[i]}^{1+\theta}, & \text{if } U_i > \sigma_{[i]}, \, |\rho| \leq \sigma_{[i]}, \\
|\rho|^{p+1} + \sigma_{[i]}^{p-1} \rho^2, & \text{if } U_i \leq \sigma_{[i]}, \, |\rho| > U_i, \\
U_i^{1+\theta} \sigma_{[i]}^{p-\theta}, & \text{otherwise}.
\end{cases}
\end{aligned}
\end{equation}
Thus, by using the \eqref{orthogonality} and \eqref{e1}, we obtain that 
\begin{equation}\label{4.111}
 \left| \int_{\mathbb{H}^n} \sigma^{p-1} \rho Z_i^{2n+2} \, d\xi \right|= \left| \int_{\mathbb{H}^n} (\sigma^{p-1}-U_i^{p-1}) \rho Z_i^{2n+2} \, d\xi \right|
\lesssim o( \varepsilon^{\frac{Q-2}{2}})+\|\rho\|_{\mathscr{D}^{1}}^2.
\end{equation}
For the third term on the right-hand side of \eqref{3-13}, we similarly have
\begin{equation}\label{NZ}
\left| N(\rho) Z_i^{2n+2} \right|
\lesssim
\begin{cases}
\sigma^{p-1} \rho^2, & \text{if } \sigma > |\rho|, \\
|\rho|^{p+1}, & \text{otherwise}.
\end{cases}
\end{equation}
Thus, combining  \eqref{3-13}--\eqref{NZ}, we deduce
\begin{equation}\label{fzi}
\left| \int_{\mathbb{H}^n} f Z_i^{2n+2} \, d\xi  \right|
\lesssim 
o( \varepsilon^{\frac{Q-2}{2}}) + \|h\|_{\mathscr{D}^{-1}} + \|\rho\|_{\mathscr{D}^{1}}^2.
\end{equation}

Without loss of generality, we assume $\lambda_1 \geq \lambda_2 \geq \cdots \geq \lambda_m$.  
For a fixed $i_0$, define
\[
\varepsilon_{i_0} := \max_{j \neq i_0} \{\varepsilon_{i_0j}\}.
\]
Then, by Lemma \ref{fz}, \eqref{e2} and   setting $i=1$ in \eqref{fzi}, we obtain
\begin{equation}\label{3.19}
\varepsilon_1^{\frac{Q-2}{2}}
\lesssim 
o( \varepsilon^{\frac{Q-2}{2}}) + \|h\|_{\mathscr{D}^{-1}} + \|\rho\|_{\mathscr{D}^{1}}^2.
\end{equation}
Next, taking $i=2$, we deduce from \eqref{fzi}, \eqref{3.19}  and Lemma \ref{fz} that
\begin{equation}
\begin{aligned}
\left| \sum_{i \neq 2} \int_{\mathbb{H}^n} U_2^{p-1} U_i Z_2^{2n+2} \right|
&\lesssim 
\left| \int_{\mathbb{H}^n} U_2^{p-1} U_1 Z_2^{2n+2} \right|
+ o( \varepsilon^{\frac{Q-2}{2}}) + \|h\|_{\mathscr{D}^{-1}} + \|\rho\|_{\mathscr{D}^{1}}^2 \\
&\lesssim 
o( \varepsilon^{\frac{Q-2}{2}}) + \|h\|_{\mathscr{D}^{-1}} + \|\rho\|_{\mathscr{D}^{1}}^2.
\end{aligned}
\end{equation}
Therefore, we conclude
\begin{equation}
\varepsilon_2^{\frac{Q-2}{2}}
\lesssim 
o( \varepsilon^{\frac{Q-2}{2}}) + \|h\|_{\mathscr{D}^{-1}} + \|\rho\|_{\mathscr{D}^{1}}^2.
\end{equation}
By induction, the desired result follows.
\end{proof}

 \subsection{The second missing estimate}
 Before giving the main estimate, we first recall the following fact:
 \begin{lemma}\label{compact}
 For non-negative function $v\in L^{\frac{Q}{2}}$, define
 \[L_v^2:=\{u\mid \|u\|_{L_v^2}^2=\int_{\mathbb{H}^n } u^2v\, d\xi<+\infty \}.\]
 Then the embedding from $\mathscr{D}^{1}$ to $L_v^{2}$ is compact.
 \end{lemma}
 \begin{proof}
The proof is essentially the same as that of Proposition A.1 in \cite{Figalli2020}, so we omit it.
 \end{proof}

As a consequence, we obtain the following result:

\begin{lemma}\label{main2}
There exist constant $\hat{\delta}=\hat{\delta}(n,m)> 0$ such that for any collection of $m$ $\hat{\delta}$-weakly interacting Jerison–Lee bubbles $\{\mathfrak{g}_{\lambda_i, \xi_i } U\}_{i=1}^m$, if $\rho$ satisfies the equation
\begin{equation}
  -\Delta_{\mathbb{H}^n} \rho - p \sigma^{p-1} \rho = \varphi, \quad (\rho, Z_{i}^{a})_{\mathscr{D}^{1}} = 0,~\forall\, 1\leq i\leq m,\ 1\leq a\leq 2n+2,
\end{equation}
then the following estimate holds:
\begin{equation}
  \|\rho\|_{\mathscr{D}^{1}} \lesssim \|\varphi\|_{\mathscr{D}^{-1}}.
\end{equation}
\end{lemma}

\begin{proof}
  Otherwise, there exist $\delta_k$-weakly interacting  Jerison-Lee bubbles $\{\mathfrak{g}_i^{(k)}U\}_{i=1}^m $ such that $\delta_k  \to 0$, $\varphi_k=o_{k}(1)$ in $\mathscr{D}^{-1}$, and $\rho_k$ with $\|\rho_k \|_{\mathscr{D}^{1}}=1$ such that
    \begin{equation}\label{3-22}
  -\Delta_{\mathbb{H}^n} \rho_k -p\sigma_k^{p-1}\rho_k =\varphi_k, ~(\rho_k, Z_{i,k}^{a})_{\mathscr{D}^{1}}=0,
  \end{equation}
  where  $$\mathfrak{g}_i^{(k)}U:=\mathfrak{g}_{\lambda_i^{(k)}, \xi_i^{(k)}} U, \quad\sigma_k= \sum_{i=1}^{m} \mathfrak{g}_i^{(k)} U, \quad Z_{i,k}^{a} = \mathfrak{g}_i^{(k)} Z^a.$$ 
  
  We claim that, up to a subsequence 
  \begin{equation}\label{3.25}
  (\mathfrak{g}_1^{(k)})^{-1}\rho_k \rightharpoonup 0 \mbox{ in } \mathscr{D}^1 \mbox{ and  } (\mathfrak{g}_1^{(k)})^{-1}\rho_k \to 0  \mbox{ in } L_{U^{p-1}}^2 \mbox{ as } k\to+\infty. 
  \end{equation}
One can verify that
    \begin{equation}\label{1.4}
    \begin{aligned}
  -\Delta_{\mathbb{H}^n}  (\mathfrak{g}_1^{(k)})^{-1} \rho_k -&p\left (U+\sum_{j> 1}   (\mathfrak{g}_1^{(k)})^{-1} \mathfrak{g}_j^{(k)} U \right)^{p-1}(\mathfrak{g}_1^{(k)})^{-1} \rho_k\\& = (\lambda_{1}^{(k)})^{-2} (\mathfrak{g}_1^{(k)})^{-1} \varphi_k=o_k(1)\in \mathscr{D}^{-1}.
  \end{aligned}
  \end{equation}
  Since the bubbles $\{(\mathfrak{g}_1^{(k)})^{-1} \mathfrak{g}_j^{(k)} U \}_{j=1}^{m}$ are also $\delta_k$-weakly interacting,
\eqref{1.4} yields that,
 without loss of generality, we can assume $\mathfrak{g}_1^{(k)}= \mathrm{id}$ for convenience when proving \eqref{3.25}. Hence the equation becomes
   \begin{equation}
    \begin{aligned}
  -\Delta_{\mathbb{H}^n}   \rho_k -&p(U+\sum_{j> 1}  \mathfrak{g}_j^{(k)}U )^{p-1} \rho_k =o_k(1)\in \mathscr{D}^{-1}.
  \end{aligned}
  \end{equation}
Because $\|\rho_k \|_{\mathscr{D}^{1}}=1$, via Lemma \ref{compact}, up to a subsequence, we assume that
$$\rho_{k} \rightharpoonup \rho \mbox{ in } \mathscr{D}^1 \mbox{ and  } \rho_k \to \rho \mbox{ in } L_{U^{p-1}}^2.$$
For  test function $\psi\in C_c^{\infty}$, we have
\begin{equation*}
\begin{aligned}
&\left| \int_{\mathbb{H}^n} \Big{[} ( U + \sum_{j > 1} \mathfrak{g}_j^{(k)} U )^{p-1} - U^{p-1} \Big{]} \rho_k \psi \, d\xi \right| \\
&\quad \lesssim \sum_{j > 1} \left( \int_{\mathbb{H}^n} U^{p-2} \, \mathfrak{g}_j^{(k)} U \, |\rho_k| \, |\psi| \, d\xi 
+ \int_{\mathbb{H}^n} (\mathfrak{g}_j^{(k)} U)^{p-1} |\rho_k| \, |\psi| \, d\xi \right).
\end{aligned}
\end{equation*}
Because these bubbles are $\delta_k$-weakly interacting with $\delta_k\to 0$, it follows that $\mathfrak{g}_j^{(k)} \rightharpoonup 0$. Thus, up to a subsequence again, we assume $$\mathfrak{g}_j^{(k)}U ,~(\mathfrak{g}_j^{(k)})^{-1}\psi \to 0 \mbox{ strongly in } L_{U^{p-1}}^2$$
Since the functions $\dfrac{|\rho_k|\, |\psi|}{U}$ and $(\mathfrak{g}_j^{(k)})^{-1} \rho_k$ are bounded in $L_{U^{p-1}}^2$, we deduce that
\[
\int_{\mathbb{H}^n} U^{p-1} \, \mathfrak{g}_j^{(k)} U \cdot \frac{|\rho_k|\, |\psi|}{U} \, d\xi \to 0,
\]
and
\begin{equation}
\int_{\mathbb{H}^n} |\mathfrak{g}_j^{(k)} U|^{p-1} |\rho_k|\, |\psi| \, d\xi 
= \int_{\mathbb{H}^n} U^{p-1} \left| (\mathfrak{g}_j^{(k)})^{-1} \rho_k \right| \cdot \left| (\mathfrak{g}_j^{(k)})^{-1} \psi \right| \, d\xi \to 0.
\end{equation}
Therefore, by letting $k \to +\infty$, we obtain
 \[-\Delta_{\mathbb{H}^n}   \rho -pU^{p-1} \rho =0 \]
 in the sense of distribution. On the other hand,
 since $$(\rho, Z^{a})_{\mathscr{D}^1}=\lim_{k\to\infty} ((\mathfrak{g}_1^{(k)})^{-1}\rho_k, Z^{a})_{\mathscr{D}^1}=\lim_{k\to\infty} (\rho_k, Z_{1,k}^{a})_{\mathscr{D}^1}=0,$$ it follows from Lemma \ref{nonde} that
  $\rho=0$. Therefore, we prove the claim \eqref{3.25}.
  
 Similarly,  up to a subsequence, for each $i$,  $$(\mathfrak{g}_i^{(k)})^{-1}\rho_k \rightharpoonup 0 \mbox{ in }   \mathscr{D}^1, \quad (\mathfrak{g}_i^{(k)})^{-1}\rho_k \to 0 \mbox{ in } L_{U^{p-1}}^2.$$
 Now test the equation \eqref{3-22} by $\rho_k$, we have
 \begin{equation}
 \begin{aligned}
 \| \rho_k\|_{\mathscr{D}^1}^2 &\lesssim \sum_{i=1}^{m}\int_{\mathbb{H}^n}  ( \mathfrak{g}_i^{(k)} U )^{p-1} \rho_k^2 \,d\xi +o_k(1)\\&
= \sum_{i=1}^{m} \int_{\mathbb{H}^n}  U^{p-1} |  (\mathfrak{g}_i^{(k)})^{-1}\rho_k|^2  \,d\xi+o_k(1)\\&
=o_k(1).
 \end{aligned}
 \end{equation}
This contradicts the fact that $ \| \rho_k\|_{\mathscr{D}^1}=1$.
  \end{proof}

We have now provided all the previously missing estimates in the proof of Theorem \ref{thm1-2}.
\section{Proof of Theorem \ref{thmsharp}}
As a consequence of Lemma \ref{compact}, for given Jerison-Lee bubbles, the operator
\[ K:= p(-\Delta_{\mathbb{H}^n })^{-1}\sigma^{p-1}: \mathscr{D}^1 \to \mathscr{D}^1   \]
is compact. Define
$$\mathscr{F} := \text{span}\left\{ Z_{i}^a \mid 1 \leq a \leq 2n+2, 1 \leq i \leq m \right\},$$
and let $\{e_{k}\}_{1 \leq k \leq m(2n+2)}$ be an orthonormal basis of $\mathscr{F}$.
Denote by \( P_{\mathscr{F}^\perp} \) (respectively \( P_{\mathscr{F}} \)) the projection from \(\mathscr{D}^1  \) onto \( \mathscr{F}^\perp \) (respectively \( \mathscr{F} \)). Then,  

\begin{lemma}\label{sharp1}
By choosing a smaller $\hat{\delta}$ in Lemma \ref{main2}, for any $\hat{\delta}$-weakly interacting Jerison–Lee bubbles, we can find a solution $\rho\in \mathscr{F}^\perp$ to equation:
\begin{equation}\label{3.45}
\begin{aligned}
  -\Delta_{\mathbb{H}^n} \rho- p\sigma^{p-1}\rho = \Delta_{\mathbb{H}^n} P_{\mathscr{F}^\perp} \Delta_{\mathbb{H}^n}^{-1} (f+N(\rho)) -p \sum_{k=1}^{m(2n+2)}\int_{\mathbb{H}^n}  \sigma^{p-1} \rho e_{k} \, d\xi  \, \cdot\Delta_{\mathbb{H}^n}e_{k}.
  \end{aligned}
  \end{equation}
In particular, we have
  \begin{equation}
  \begin{aligned}
   \|\rho\|_{\mathscr{D}^1}^2\approx \|P_{\mathscr{F}^\perp} \Delta_{\mathbb{H}^n}^{-1} f\|_{\mathscr{D}^1}^2
=\| f\|_{\mathscr{D}^{-1}}^2-\|P_{\mathscr{F}} \Delta_{\mathbb{H}^n}^{-1} f\|_{\mathscr{D}^1}^2 .
   \end{aligned}
  \end{equation}
\end{lemma}
\begin{proof}
We first prove that the  Fredholm operator
 $$\mathrm{id}-P_{\mathscr{F}^\perp}K:   \mathscr{F}^\perp\to\mathscr{F}^\perp $$ 
 is uniformly bounded and invertible; that is, there exists a constant $C_0>0$, independent of the choice of  $\hat{\delta}$-weakly interacting Jerison–Lee bubbles, such that 
 \begin{equation}\label{3.22222}
   \frac{1}{C_0}\leq\|(\mathrm{id}-P_{\mathscr{F}^\perp}K)^{-1}\|_{\mathscr{L}(\mathscr{F}^\perp,\mathscr{F}^\perp )} \leq C_0.
 \end{equation}
The boundedness is obvious, so we focus on proving uniform invertibility. 
Suppose $\rho, \psi\in \mathscr{F}^\perp$ satisfy
\[(\mathrm{id}-P_{\mathscr{F}^\perp}K)\rho=\psi. \] 
Then, we have
\[ -\Delta_{\mathbb{H}^n} \rho- p\sigma^{p-1}\rho = -\Delta_{\mathbb{H}^n} \psi  -p \sum_{k=1}^{m(2n+2)}\Big{(}\int_{\mathbb{H}^n}  \sigma^{p-1} \rho e_{k} \, d\xi  \, \Big{)}\Delta_{\mathbb{H}^n}e_{k}. \]
By Lemma~\ref{main2}, it follows that 
\begin{equation}
\|\rho\|_{\mathscr{D}^1 }\lesssim  \|\Delta_{\mathbb{H}^n} \psi\|_{\mathscr{D}^{-1} }+
\sum_{k=1}^{m(2n+2)} |\int_{\mathbb{H}^n}  \sigma^{p-1} \rho e_{k} \, d\xi |
\end{equation}
Since $\{e_{k}\}$ can be obtained by Schmidt orthogonalization of \( \{ Z_i^{a} \} \), and $ \rho\in \mathscr{F}^\perp$, we have
\begin{equation}\label{3-36}
\begin{aligned}
\sum_{k=1}^{m(2n+2)} |\int_{\mathbb{H}^n} \sigma^{p-1} \rho e_{k} \, d\xi| &\approx \sum_{i,a} |\int_{\mathbb{H}^n} (\sigma^{p-1}-Q_i^{p-1}) \rho Z_i^a \, d\xi|\\&
\lesssim
o_{\hat{\delta}}( \|\rho\|_{\mathscr{D}^1}).
\end{aligned}
\end{equation}
Therefore, for $\hat{\delta}$ sufficiently small, there exists a constant $C_0 > 0$, independent of the bubble configuration, such that
\begin{equation}
\|\rho\|_{\mathscr{D}^1 }\leq   C_0\| \psi\|_{\mathscr{D}^{1} }.
\end{equation}
Finally, Fredholm’s alternative implies the uniform invertibility.

Observe that equation \eqref{3.45} is equivalent to solving
  \begin{equation}\label{3-29}
  (\mathrm{id}-P_{\mathscr{F}^\perp}K)\rho=-P_{\mathscr{F}^\perp}\Delta_{\mathbb{H}^n}^{-1} (f+N(\rho)) \mbox{ in } \mathscr{F}^\perp.
  \end{equation}
By Lemma \ref{fh-1}, we have the estimate
$$\|P_{\mathscr{F}^\perp} \Delta_{\mathbb{H}^n}^{-1} f\|_{\mathscr{D}^1}\lesssim\|f\|_{\mathscr{D}^{-1}}\leq C_1\hat{\delta} ,$$
for some constant $C_1>0$. Recall that the nonlinearity $N(\rho)$ satisfies
\begin{equation*}
   |N(\rho)| \lesssim
    \begin{cases}
    \sigma^{p-2}\rho^2+|\rho|^p, & \mbox{if } p>2 ,\\
    |\rho|^p, & \mbox{if } 1<p\leq2,
  \end{cases}
  \end{equation*}
  which yields
  \begin{equation}\label{Np}
  \|P_{\mathscr{F}^\perp} \Delta_{\mathbb{H}^n}^{-1} N(\rho)\|_{\mathscr{D}^1}\leq C_2 \|\rho\|_{\mathscr{D}^1}^{\min\{p,2\}},
  \end{equation}
for some contant $C_2 >0$. 
We define the operator $$\mathrm{T}(\rho):= -(\mathrm{id}-P_{\mathscr{F}^\perp}K)^{-1} P_{\mathscr{F}^\perp} \Delta_{\mathbb{H}^n}^{-1} (f+N(\rho)).$$
Then, there exists a sufficiently small quantity   $c(\hat{\delta})=o_{\hat{\delta}}(1)$ such that $$C_0C_1\hat{\delta}+C_0C_2 c(\hat{\delta})^{\min\{p,2\}} <c(\hat{\delta}),$$ 
which ensures that the mapping $\mathrm{T}$ leaves the ball
$$\mathrm{B}=\{\rho\in \mathscr{F}^\perp \mid \|\rho\|_{\mathscr{D}^1}\leq c(\hat{\delta})\}$$
invariant; that is, $\mathrm{T}(\mathrm{B})\subset \mathrm{B}$. Moreover,
for $\rho_1, ~\rho_2\in \mathrm{B},$ we have
\begin{equation}
\begin{aligned}
 \|\mathrm{T}(\rho_1)-\mathrm{T}(\rho_2) \|_{\mathscr{D}^1}&\lesssim  \|   N(\rho_1)-N(\rho_2)\|_{\mathscr{D}^{-1}}\\&
 \lesssim
 \|   N(\rho_1)-N(\rho_2)\|_{L^{\frac{2Q}{Q+2}}}\\&
 \lesssim  o_{\hat{\delta}}(1)\| \rho_1-\rho_2 \|_{\mathscr{D}^1} .
 \end{aligned}
\end{equation}
Thus, by applying the contraction mapping principle, we can solve equation \eqref{3-29}.  In particular, 
\begin{equation}
\begin{aligned}
\|\rho\|_{\mathscr{D}^1} &\approx  \|P_{\mathscr{F}^\perp} \Delta_{\mathbb{H}^n}^{-1} (f+N(\rho))\|_{\mathscr{D}^1}\\&
\approx  \|P_{\mathscr{F}^\perp} \Delta_{\mathbb{H}^n}^{-1} f\|_{\mathscr{D}^1}.
\end{aligned}
\end{equation}
\end{proof}

Now we consider the function $\rho$ obtained from Lemma \ref{sharp1}. By Lemma \ref{fh-1}, we have
\begin{equation}\label{3-33}
 \|\rho\|_{\mathscr{D}^1}\lesssim \|f\|_{\mathscr{D}^{-1}}
  \lesssim \begin{cases}
             \varepsilon^{\frac{Q+2}{4}}, & \mbox{if } n\geq 3,\\
             \varepsilon^2|\log \varepsilon|^{\frac{1}{2}}, & \mbox{if } n=2,\\
              \varepsilon, & \mbox{if } n=1.
           \end{cases}
  \end{equation}

 Furthermore, the following proposition holds.
\begin{proposition}\label{999}
Let $\sigma(z,t)= U(z,t)+U(z,t+\frac{1}{\varepsilon})$ with $0<\varepsilon<\hat{\delta}$. For $\rho$ obtained from Lemma \ref{sharp1}, the function $u=\sigma+\rho$ satisfies
  \begin{equation}
\|\rho\|_{\mathscr{D}^1}\gtrsim
\begin{cases}
\|\Delta_{\mathbb{H}^n}u+|u|^{\frac{2}{n}}u \|_{\mathscr{D}^{-1}}|\log  \|\Delta_{\mathbb{H}^n}u+|u|^{\frac{2}{n}}u \|_{\mathscr{D}^{-1}}|^{\frac{1}{2}}, & \mbox{if } n=2 ,\\
\|\Delta_{\mathbb{H}^n}u+|u|^{\frac{2}{n}}u \|_{\mathscr{D}^{-1}}^{\frac{n+2}{2n}}, & \mbox{if } n\geq 3.
\end{cases}
\end{equation}
\end{proposition}
\begin{proof}
Direct calculation yields
\begin{equation}\label{sharpu}
\begin{aligned}
  -\Delta_{\mathbb{H}^n} u- |u|^{p-1}u=- \Delta_{\mathbb{H}^n}P_{\mathscr{F}}\Delta_{\mathbb{H}^n}^{-1} (f+N(\rho)) -p \sum_{k=1}^{m(2n+2)}\Big{(}\int_{\mathbb{H}^n}  \sigma^{p-1} \rho e_{k} \, d\xi  \, \Big{)}\Delta_{\mathbb{H}^n}e_{k}.
\end{aligned}
\end{equation}
As shown in equation \eqref{4.111}, we have
\begin{equation}\label{3-36}
\begin{aligned}
\sum_{k=1}^{m(2n+2)} |\int_{\mathbb{H}^n} \sigma^{p-1} \rho e_{k} \, d\xi| &\approx \sum_{i,a} |\int_{\mathbb{H}^n} (\sigma^{p-1}-Q_i^{p-1}) \rho Z_i^a \, d\xi|\\&
\lesssim
o(\varepsilon^{\frac{Q-2}{2}})+ \|\rho\|_{\mathscr{D}^1}^2.
\end{aligned}
\end{equation}
On the other hand, \eqref{NZ} and \eqref{3-33} yield
\begin{equation}\label{3-37}
\begin{aligned}
 \|\Delta_{\mathbb{H}^n} P_{\mathscr{F}}\Delta_{\mathbb{H}^n}^{-1} &(f+N(\rho))\|_{\mathscr{D}^{-1}}
 = \|P_{\mathscr{F}}\Delta_{\mathbb{H}^n}^{-1} (f+N(\rho))\|_{\mathscr{D}^{1}}\\&
 \approx \sum_{i, a} \left|\int_{\mathbb{H}^n} (f+N(\rho))Z_i^a \, d\xi\right|\\&
 \lesssim  \sum_{i, a} \left |\int_{\mathbb{H}^n}  f Z_i^a \, d\xi\right|+\|\rho\|_{\mathscr{D}^1}^2\\&
 \lesssim \sum\limits_{\substack{i,j \\ i\neq j}} \left |\int_{\mathbb{H}^n}  U_i^p U_j \, d\xi\right|+\|\rho\|_{\mathscr{D}^1}^2\\&
  \approx
\varepsilon^{\frac{Q-2}{2}}.
\end{aligned}
\end{equation}
Therefore, we must have
\begin{equation}\label{sha1}
\begin{aligned}
\|\Delta_{\mathbb{H}^n} u+ |u|^{p-1}u\|_{\mathscr{D}^{-1}}
\lesssim \varepsilon^{\frac{Q-2}{2}}.
\end{aligned}
\end{equation}
Combining Lemma \ref{sharpes} with Lemma \ref{sharp1}, we obtain
\begin{equation}\label{sha}
\begin{aligned}
\|\rho\|_{\mathscr{D}^1}&\gtrsim\| f\|_{\mathscr{D}^{-1}}- C\varepsilon^{\frac{Q-2}{2}}\\& \gtrsim
\begin{cases}
\varepsilon^2|\log  \varepsilon|^{\frac{1}{2}}, & \mbox{if } n=2,\\
  \varepsilon^{\frac{Q+2}{4}}, & \mbox{if } n\geq 3
\end{cases}\\& \gtrsim
\begin{cases}
\|\Delta_{\mathbb{H}^n}u+|u|^{\frac{2}{n}}u \|_{\mathscr{D}^{-1}}|\log  \|\Delta_{\mathbb{H}^n}u+|u|^{\frac{2}{n}}u \|_{\mathscr{D}^{-1}}|^{\frac{1}{2}}, & \mbox{if } n=2,\\
\|\Delta_{\mathbb{H}^n}u+|u|^{\frac{2}{n}}u \|_{\mathscr{D}^{-1}}^{\frac{n+2}{2n}}, & \mbox{if } n\geq 3.
\end{cases}
\end{aligned}
\end{equation}
\end{proof}

\begin{proof}[Proof of Theorem \ref{thmsharp}]
Due to Proposition \ref{999}, to verify the sharpness, it suffices to prove that
 \begin{equation}\label{rhosharp}
  \inf_{\lambda_i, \xi_i}\| u-\sum_{i=1}^{m}\mathfrak{g}_{\lambda_i, \xi_i }U \|_{\mathscr{D}^1}\gtrsim \|\rho\|_{\mathscr{D}^1}, 
  \end{equation}
  which follows from the Taylor expansion and the fact $\rho \in \mathscr{F}^{\perp}$, see \cite[Section 4.5]{Figalli2020} and \cite[page 219]{Wei} for instance. Furthermore, we shall demonstrate that $u^{+}$ remains a sharp example.
  By testing the equation \eqref{sharpu} with $u^{-}$, and using \eqref{3-33}, \eqref{3-36}, and \eqref{3-37}, we deduce that
  \[
\begin{aligned}
\|u^{-}\|_{\mathscr{D}^1}^2\lesssim\| u^{-}\|_{\mathscr{D}^1}^{p+1}+\varepsilon^{\frac{Q-2}{2}}\| u^{-}\|_{\mathscr{D}^1}.
\end{aligned}
\]
Thus, we obtain $$\|u^{-}\|_{\mathscr{D}^1}\lesssim \varepsilon^{\frac{Q-2}{2}}.$$
Next, we estimate the difference
\[
\begin{aligned}
 \Big{|}\| \Delta_{\mathbb{H}^n} u^+ &+ |u^+|^{p-1} u^+ \|_{\mathscr{D}^{-1}} - \| \Delta_{\mathbb{H}^n} u + |u|^{p-1} u \|_{\mathscr{D}^{-1}}  \Big{|}\\
& \leq \| u^{-} \|_{\mathscr{D}^1} + \| |u^+|^{p-1} u^+ - |u|^{p-1} u \|_{\mathscr{D}^{-1}} \\
& \lesssim \| u^{-} \|_{\mathscr{D}^1} + \| u^{-} \|_{\mathscr{D}^1}^p \\
& \lesssim \varepsilon^{\frac{Q-2}{2}}.
\end{aligned}
\]
Therefore, by \eqref{sha1}, \eqref{sha} and \eqref{rhosharp}, we have
\begin{equation*}
\begin{aligned}
\inf_{\lambda_i, \xi_i}&\| u^{+}-\sum_{i=1}^{m}\mathfrak{g}_{\lambda_i, \xi_i }U \|_{\mathscr{D}^1}\\&\gtrsim \inf_{\lambda_i, \xi_i}\| u-\sum_{i=1}^{m}\mathfrak{g}_{\lambda_i, \xi_i }U \|_{\mathscr{D}^1}-C\varepsilon^{\frac{Q-2}{2}}\\&\gtrsim
\begin{cases}
\varepsilon^2|\log  \varepsilon|^{\frac{1}{2}}, & \mbox{if } n=2\\
  \varepsilon^{\frac{Q+2}{4}}, & \mbox{if } m\geq 3
\end{cases}\\& \gtrsim
\begin{cases}
\|\Delta_{\mathbb{H}^n} u^{+}+| u^{+}|^{\frac{2}{n}} u^{+} \|_{\mathscr{D}^{-1}}|\log  \|\Delta_{\mathbb{H}^n} u^{+}+| u^{+}|^{\frac{2}{n}} u^{+} \|_{\mathscr{D}^{-1}}|^{\frac{1}{2}}, & \mbox{if } n=2 ,\\
\|\Delta_{\mathbb{H}^n} u^{+}+| u^{+}|^{\frac{2}{n}} u^{+} \|_{\mathscr{D}^{-1}}^{\frac{n+2}{2n}}, & \mbox{if } n\geq 3.
\end{cases}
\end{aligned}
\end{equation*}

\end{proof}

\section{Appendix}
In this section, we carry out a detailed computation of the interaction integrals between a standard bubble $U$ and another bubble $\mathfrak{g}_{\lambda, \xi_0}U$ with a scaling parameter $\lambda \leq 1$. Owing to the  invariance property of the gauge transformation $\mathfrak{g}$, as discussed in Section 2.2, the interaction between two Jerison–Lee bubbles can, without loss of generality, be reduced to  the canonical configuration under consideration.

Note that in this case, 
$$\varepsilon=\min\{\frac{1}{ \lambda |\xi_0|^2}, \lambda\}\leq1.$$
The following estimates hold to be true:
  \begin{lemma}\label{app1}
    Let $\lambda\leq 1$, then for any positive $\alpha\neq \beta$ such that
    $\alpha+\beta=2_Q^*$, we have
   \begin{equation}
   \int_{\mathbb{H}^n}U^{\alpha}(\xi)|\mathfrak{g}_{\lambda, \xi_0}U(\xi)|^{\beta} \, d\xi \approx  \varepsilon^{\frac{\min\{\alpha,\beta\}(Q-2)}{2}}.
   \end{equation}
  \end{lemma}
 \begin{proof}
If $\varepsilon=\lambda$, then $\lambda |\xi_0|\leq 1 $.
   In the region $B_{\frac{2}{\lambda}}$ we have $$\mathfrak{g}_{\lambda, \xi_0}U \approx \lambda^{\frac{Q-2}{2}},~ U \approx \frac{1}{1+|\xi|^{Q-2}};$$
    in $B_{\frac{2}{\lambda}}^c$ we have $$\mathfrak{g}_{\lambda, \xi_0}U \approx \frac{1}{\lambda^{\frac{Q-2}{2}}|\xi|^{Q-2} },~ U \approx \frac{1}{|\xi|^{Q-2}}.$$ Thus, we can deduce that
   \begin{equation}
   \begin{aligned}
      \int_{\mathbb{H}^n}& U^{\alpha}(\xi)|\mathfrak{g}_{\lambda, \xi_0}U(\xi)|^{\beta} \, d\xi \\
      &\approx \int_{0}^{\frac{2}{\lambda}}  \lambda^{\frac{\beta(Q-2)}{2}} \frac{r^{Q-1}}{1+r^{\alpha(Q-2)}} dr+\int_{\frac{2}{\lambda}}^{+\infty} \frac{r^{Q-1}}{\lambda^{\frac{\beta(Q-2)}{2}} r^{(Q-2)(\alpha+\beta)}  }dr \\&\approx  \varepsilon^{\frac{\min\{\alpha,\beta\}(Q-2)}{2}}.
   \end{aligned}
   \end{equation}
   
   If $\varepsilon=\frac{1}{ \lambda |\xi_0|^2}$, then $\lambda |\xi_0|\geq 1 $. In the region $B_{\frac{|\xi_0|}{2}}(\xi_0)$, 
   $$U\approx \frac{1}{|\xi_0|^{Q-2}},~\mathfrak{g}_{\lambda, \xi_0}U \approx  \frac{\lambda^{\frac{Q-2}{2}}}{(1+\lambda |\xi_0^{-1}\circ \xi|)^{Q-2}};$$ in the region $B_{\frac{|\xi_0|}{2}}^c(\xi_0) \cap B_{\frac{|\xi_0|}{2}}(0)$, $$U \approx \frac{1}{1+|\xi|^{Q-2}}, \quad \mathfrak{g}_{\lambda, \xi_0}U \approx  \frac{1}{\lambda^{\frac{Q-2}{2}} |\xi_0|^{Q-2}};$$ in the region $B_{\frac{|\xi_0|}{2}}^c(\xi_0) \cap B_{\frac{|\xi_0|}{2}}^c(0)$, $$U\approx \frac{1}{|\xi|^{Q-2}},~\mathfrak{g}_{\lambda, \xi_0}U \approx  \frac{1}{\lambda^{\frac{Q-2}{2}} |\xi|^{Q-2}}.$$ 
   Then, we have
   \begin{equation}
   \begin{aligned}
      \int_{\mathbb{H}^n}U^{\alpha}(\xi)&|\mathfrak{g}_{\lambda, \xi_0}U|^{\beta}(\xi) \, d\xi 
      \approx |\xi_0|^{-\alpha (Q-2)}  \int_{0}^{\frac{|\xi_0|}{2}} \frac{\lambda^{\frac{\beta(Q-2)}{2}} r^{Q-1}}{(1+\lambda r)^{\beta(Q-2)}} dr\\&
      +  \frac{\int_{0}^{\frac{|\xi_0|}{2}} \frac{r^{Q-1}}{1+r^{\alpha(Q-2)}}dr}{\lambda^{\frac{\beta(Q-2)}{2}} |\xi_0|^{\beta(Q-2)}}
      +\int_{\frac{|\xi_0 |}{2}}^{+\infty} \frac{r^{Q-1}}{\lambda^{\frac{\beta(Q-2)}{2}} r^{(Q-2)(\alpha+\beta)}  }dr \\
      \approx&  \varepsilon^{\frac{\min\{\alpha,\beta\}(Q-2)}{2}}.
   \end{aligned}
   \end{equation}
 \end{proof}

  \begin{lemma}\label{app2}
Let $\lambda\leq 1$,  $\alpha>\beta>0$ with $\alpha+\beta=2_Q^*$. Then for any positive bounded function $K(\xi)$, we have
   \begin{equation}
   \int_{\mathbb{H}^n}K(\xi) U^{\alpha}(\xi)|\mathfrak{g}_{\lambda, \xi_0}U(\xi)|^{\beta} \, d\xi=
     \frac{ c_0^{p+1}c_{\alpha}    \lambda^{\frac{\beta(Q-2)}{2}}}{ [(1+\lambda^2 |z_0|^2 )^2  + \lambda^4 t_0^2]^{\frac{\beta(Q-2)}{4}}}+o(\varepsilon^{\frac{\beta(Q-2)}{2}}),
   \end{equation}
   where $c_\alpha$ is the constant given by $$c_\alpha=\int_{\mathbb{H}^n} K(\xi)U^{\alpha}(\xi) \, d\xi. $$
  \end{lemma}
  \begin{proof}
  The range of $\alpha, ~\beta$ implies that $$\alpha(Q-2)>Q>\beta(Q-2).$$ Denote by $$I(\xi )=K(\xi) U^{\alpha}(\xi)|\mathfrak{g}_{\lambda, \xi_0}U(\xi)|^{\beta} .$$
  
   If $\varepsilon=\lambda$, we use
  \begin{equation}
       \int_{\mathbb{H}^n}I(\xi ) \, d\xi =\int_{B_{\frac{1}{100\lambda }}}I(\xi ) \, d\xi +\int_{B_{\frac{1}{100\lambda }}^c}I(\xi ) \, d\xi .
    \end{equation}
  In the region $B_{\frac{1}{100\lambda }}^c$, by direct calculation
    \[ \int_{B_{\frac{1}{100\lambda }}^c} I(\xi)  \, d\xi \lesssim  \lambda^{\frac{\beta(Q-2)}{2}} \int_{B_{\frac{1}{100\lambda }}^c}  \frac{1}{[(1+|z|^2)^2 +t^2]^{\frac{\alpha (Q-2)}{4}}} \, d\xi\lesssim  o(\lambda^{\frac{\beta(Q-2)}{2}}). \]
    In the region $B_{\frac{1}{100\lambda }}$, we have
    \begin{equation}\label{case1.1}
    \begin{aligned}
      &(1+\lambda^2|z-z_0|^2)^2 +\lambda^4(t-t_0-2 \text{Im} (\overline{z}\cdot z_0) )^2
     \\& =[(1+\lambda^2 |z_0|^2 )^2  + \lambda^4 t_0^2]+[2\lambda^2(1+\lambda^2 |z_0|^2 )(|z|^2-2\text{Re} (\overline{z}\cdot z_0))]\\&
     +\lambda^4(|z|^2-2\text{Re} (\overline{z}\cdot z_0))^2 +\lambda^4(t-2\text{Im} (\overline{z}\cdot z_0))^2+2\lambda^4t_0(2 \text{Im} (\overline{z}\cdot z_0)-t)\\&
     =[(1+\lambda^2 |z_0|^2 )^2  + \lambda^4 t_0^2](1+P(\lambda z, \lambda^2 t)).
      \end{aligned}
    \end{equation}
    Here $P$ is a polynomial with uniformly bounded coefficients such that $P(0)=0$ and $|P(\lambda z, \lambda^2 t)|\leq \frac{1}{2}$ in $B_{\frac{1}{100\lambda }}$. Now, combining the fact that
    $$\Big{|}\frac{1}{(1+P(\lambda z, \lambda^2 t))^\frac{\beta(Q-2)}{4}}-1\Big{|}\lesssim  P(\lambda z, \lambda^2 t),$$
the assumption $\alpha(Q-2)>Q$,  and
    \[  \int_{B_{\frac{1}{100\lambda }}}  \frac{\lambda^k |\xi|^k }{[(1+|z|^2)^2 +t^2]^{\frac{\alpha (Q-2)}{4}}} \, d\xi \lesssim \lambda^{\alpha(Q-2)-Q}| \log \lambda|=o(1) ,   ~~\forall k\geq 0,\]
    we conclude that
    \begin{equation}\label{4-7}
    \begin{aligned}
     \int_{\mathbb{H}^n}I(\xi ) \, d\xi &=\int_{B_{\frac{1}{100\lambda }}}I(\xi ) \, d\xi+ o(\lambda^{\frac{\beta(Q-2)}{2}})\\
      &= \frac{c_0^{p+1}\lambda^{\frac{\beta(Q-2)}{2}} }{[(1+\lambda^2 |z_0|^2 )^2  + \lambda^4 t_0^2]^{\frac{\beta(Q-2)}{4}}} \Big{[}\int_{B_{\frac{1}{100\lambda }}} \frac{K(\xi)}{[(1+|z|^2)^2 +t^2]^{\frac{\alpha(Q-2)}{4}} } \, d\xi  +o(1)\Big{]}\\&
     = \frac{c_0^{p+1}\lambda^{\frac{\beta(Q-2)}{2}} }{[(1+\lambda^2 |z_0|^2 )^2  + \lambda^4 t_0^2]^{\frac{\beta(Q-2)}{4}}} \Big{[}\int_{\mathbb{H}^n} \frac{K(\xi)}{[(1+|z|^2)^2 +t^2]^{\frac{\alpha(Q-2)}{4}} } \, d\xi  +o(1)\Big{]}.
     \end{aligned}
    \end{equation}

    If $\varepsilon=\frac{1}{\lambda |\xi_0|^2}$, we use
    \begin{equation}
       \int_{\mathbb{H}^n} I(\xi) \, d\xi=\int_{B_{\frac{\sqrt{\lambda} |\xi_0|}{100 }}} I(\xi) \, d\xi+\int_{B_{\frac{\sqrt{\lambda} |\xi_0|}{100 }}^c} I(\xi) \, d\xi.
    \end{equation}
 For the region $B_{\frac{\sqrt{\lambda} |\xi_0|}{100}}^c $, we divide it by 
 \[ \int_{B_{\frac{\sqrt{\lambda} |\xi_0|}{100 }}^c} I(\xi) \, d\xi= \int_{B_{\frac{\sqrt{\lambda} |\xi_0|}{100 }}^c \cap B_{\frac{|\xi_0|}{4}}(\xi_0) } I(\xi) \, d\xi +\int_{B_{\frac{\sqrt{\lambda} |\xi_0|}{100}}^c \cap B_{\frac{|\xi_0|}{4}}^c(\xi_0) }I(\xi) \, d\xi.\]
 By direct calculation, we have
 \begin{align*}
  \int_{B_{\frac{\sqrt{\lambda} |\xi_0|}{100 }}^c \cap B_{\frac{|\xi_0|}{4}}(\xi_0) } I(\xi) \, d\xi
 &\lesssim | \xi_0|^{-\alpha(Q-2)}\int_{B_{\frac{|\xi_0|}{4}}} \frac{1}{(\sqrt{\lambda}|\xi|)^{\beta(Q-2)}} \, d\xi\\
 &\lesssim |\xi_0|^{Q-\alpha(Q-2)} \frac{1}{(\sqrt{\lambda}|\xi_0|)^{\beta(Q-2)}}\\&
 =o(\varepsilon^{\frac{\beta(Q-2)}{2}}),
 \end{align*}
 and
 \begin{align*}
  \int_{B_{\frac{\sqrt{\lambda} |\xi_0|}{100 }}^c \cap B_{\frac{|\xi_0|}{4}}^c(\xi_0) } I(\xi)\, d\xi& \lesssim  (\sqrt{\lambda} |\xi_0|)^{-\beta(Q-2)}  \int_{B_{\frac{\sqrt{\lambda} |\xi_0|}{100}}^c } \frac{1}{|\xi|^{\alpha(Q-2)}}  \, d\xi\\
 &\lesssim  (\sqrt{\lambda}|\xi_0|)^{Q-\alpha(Q-2)-\beta(Q-2)} \\&
 =o(\varepsilon^{\frac{\beta(Q-2)}{2}}) .
 \end{align*}
 On the other hand, in the region $B_{\frac{\sqrt{\lambda} |\xi_0|}{100 }}$, we have 
   \begin{equation}
   \begin{aligned} &\frac{\lambda^{\frac{Q-2}{2}}}{[(1+\lambda^2|z-z_0|^2)^2 +\lambda^4(t-t_0-2 \text{Im}(\overline{z}\cdot z_0) )^2]^{\frac{Q-2}{4}}}
   \\&=\frac{1}{[(\frac{1}{\lambda}+\lambda|z-z_0|^2)^2 +\lambda^2(t-t_0-2 \text{Im} (\overline{z}\cdot z_0) )^2]^{\frac{Q-2}{4}}},
   \end{aligned}
   \end{equation}
   and   
\begin{equation}\label{case1.2}
    \begin{aligned}
      &(\frac{1}{\lambda}+\lambda|z-z_0|^2)^2 +\lambda^2(t-t_0-2 \text{Im} (\overline{z}\cdot z_0) )^2
     \\& =[(\frac{1}{\lambda}+\lambda|z_0|^2 )^2  + \lambda^2 t_0^2]+[2\lambda(\frac{1}{\lambda}+\lambda |z_0|^2 )(|z|^2-2\text{Re} (\overline{z}\cdot z_0))]\\&
     +\lambda^2(|z|^2-2\text{Re} (\overline{z}\cdot z_0))^2 +\lambda^2(t-2\text{Im} (\overline{z}\cdot z_0))^2+2\lambda^2t_0(2 \text{Im} (\overline{z}\cdot z_0)-t)\\&
     =[(\frac{1}{\lambda}+\lambda |z_0|^2 )^2  + \lambda^2 t_0^2](1+Q(\frac{z}{\sqrt{\lambda} |\xi_0|} , \frac{t}{\lambda |\xi_0|^2})),
      \end{aligned}
    \end{equation}
     where $Q$ is a polynomial with uniformly bounded coefficients such that $Q(0)=0$ and $|Q(\frac{z}{\sqrt{\lambda} |\xi_0|} , \frac{t}{\lambda |\xi_0|^2})|\leq \frac{1}{2}$ in $B_{\frac{\sqrt{\lambda} |\xi_0|}{100}}$.
   Therefore, similar to \eqref{4-7}, we obtain
   \begin{equation}\label{4-11}
   \int_{\mathbb{H}^n }I(\xi) \, d\xi=  \frac{c_0^{p+1}}{[(\frac{1}{\lambda}+\lambda |z_0|^2 )^2  + \lambda^2 t_0^2]^{\frac{\beta(Q-2)}{4}}} \Big{[} \int_{\mathbb{H}^n} \frac{K(\xi)}{[(1+|z|^2)^2 +t^2]^{\frac{\alpha(Q-2)}{4}} } \, d\xi +o(1)\Big{]}.
   \end{equation}
   We complete the proof with \eqref{4-7} and \eqref{4-11}.
  \end{proof}

  Now, for giving gauge transformation $\mathfrak{g}_{\lambda ,\xi_0}\in \mathscr{G}$, we 
  set $$V=\mathfrak{g}_{\lambda ,\xi_0}U,~f=(U+V)^p-U^p- V^p,$$ then we have:
  \begin{lemma}\label{app3}
  Let $\lambda\leq 1$. For $Q>6$,  we have
   \begin{equation}\label{4-12}
   \int_{\mathbb{H}^n } f^{\frac{2Q}{Q+2}} \, d\xi \lesssim  \varepsilon^{\frac{Q}{2}}.
   \end{equation}
  \end{lemma}
  \begin{proof}
  Since $$V(\xi)\approx  \frac{1}{(\frac{1}{\sqrt{\lambda}} +\sqrt{\lambda} |\xi_0^{-1}\circ \xi| )^{Q-2}}, $$  in the region $B_{\frac{1}{2\sqrt{\varepsilon}}}$,   we have  $$U\gtrsim \varepsilon^{\frac{Q-2}{2}}   \approx V,
   \mbox{ so that } f\approx U^{p-1}V.$$  It follows that
    \[ \int_{ B_{\frac{1}{2\sqrt{\varepsilon}}}} f^{\frac{2Q}{Q+2}} \, d\xi \approx  \varepsilon^{\frac{Q(Q-2)}{Q+2}} \int_{B_{\frac{1}{2\sqrt{\varepsilon}}}}\frac{1}{1+|\xi|^{\frac{8Q}{Q+2}}} \, d\xi \approx \varepsilon^{\frac{Q}{2}}.  \]
   
    If $\varepsilon=\lambda$, since  $V\lesssim \varepsilon^{\frac{Q-2}{2}}$, then
    \begin{equation}\label{4-13}
    \int_{ B_{\frac{1}{2\sqrt{\varepsilon}}}^c} f^{\frac{2Q}{Q+2}} \, d\xi \lesssim  \int_{ B_{\frac{1}{2\sqrt{\varepsilon}}}^c} (U^{2_Q^*} +U^{\frac{2Q}{Q+2}} V^{\frac{8Q}{(Q+2)(Q-2)}} ) \, d\xi\lesssim \varepsilon^{\frac{Q}{2}}.
    \end{equation}
    If $\varepsilon=\frac{1}{\lambda |\xi_0|^2}$, we divide $B_{\frac{1}{2\sqrt{\varepsilon}}}^c$ into 
    $B_{\frac{1}{2\sqrt{\varepsilon}}}^c\cap B_{\frac{|\xi_0|}{4}} (\xi_0)$ and $B_{\frac{1}{2\sqrt{\varepsilon}}}^c\cap B_{\frac{|\xi_0|}{4}}^c (\xi_0).$
    In the region $B_{\frac{1}{2\sqrt{\varepsilon}}}^c\cap B_{\frac{|\xi_0|}{4}} (\xi_0)$, we have $$V\gtrsim U, \mbox{ thus } f\approx V^{p-1}U,$$ and
    we can deduce that
       \begin{equation}
    \begin{aligned}
    \int_{ B_{\frac{1}{2\sqrt{\varepsilon}}}^c\cap B_{\frac{|\xi_0|}{4}} (\xi_0)}  f^{\frac{2Q}{Q+2}} \, d\xi &\approx
    \int_{ B_{\frac{1}{2\sqrt{\varepsilon}}}^c\cap B_{\frac{|\xi_0|}{4}} (\xi_0)} U^{\frac{2Q}{Q+2}} V^{\frac{8Q}{(Q+2)(Q-2)}} \, d\xi\\
    &\lesssim |\xi_0|^{-\frac{2Q(Q-2)}{(Q+2)}} \int_{  B_{\frac{|\xi_0|}{4}}} \frac{1}{(\sqrt{\lambda}|\xi|)^{\frac{8Q}{(Q+2)}}} \, d\xi
    \\&\lesssim
    \lambda^{-\frac{4Q}{Q+2}}|\xi_0|^{-Q}\\&
    \lesssim  \varepsilon^{\frac{Q}{2}}.
    \end{aligned}
    \end{equation}
   In the region $B_{\frac{1}{2\sqrt{\varepsilon}}}^c\cap B_{\frac{|\xi_0|}{4}}^c (\xi_0)$,  we have
   
   \begin{equation}\label{4-15}
   \begin{aligned}
   \int_{ B_{\frac{1}{2\sqrt{\varepsilon}}}^c\cap B_{\frac{|\xi_0|}{4}}^c (\xi_0)} f^{\frac{2Q}{Q+2}} \, d\xi &\lesssim \int_{ B_{\frac{1}{2\sqrt{\varepsilon}}}^c\cap B_{\frac{|\xi_0|}{4}}^c (\xi_0)} (U^{2_Q^*} +U^{\frac{2Q}{Q+2}} V^{\frac{8Q}{(Q+2)(Q-2)}} ) \, d\xi\\&\lesssim \varepsilon^{\frac{Q}{2}}+ \int_{ B_{\frac{1}{2\sqrt{\varepsilon}}}^c\cap B_{\frac{|\xi_0|}{4}}^c (\xi_0)} U^{\frac{2Q}{Q+2}} V^{\frac{8Q}{(Q+2)(Q-2)}} \, d\xi\\
    &\lesssim     \varepsilon^{\frac{Q}{2}}+  \varepsilon^{\frac{4Q}{Q+2}}  \int_{ B_{\frac{1}{2\sqrt{\varepsilon}}}^c} \frac{1}{|\xi|^{\frac{2Q(Q-2)}{Q+2}}} \, d\xi
    \lesssim  \varepsilon^{\frac{Q}{2}}.
    \end{aligned}
    \end{equation}
Therefore, \eqref{4-12} follows from \eqref{4-13}-\eqref{4-15}.
  \end{proof}

  In the case $Q=6$, it holds that  $f= 2UV$. 
  And we have the following:
  \begin{lemma}\label{app4}
  Let $\lambda\leq 1$. For $Q=6$, we have
  \[ \|f\|_{\mathscr{D}^{-1}}^2 \lesssim \varepsilon^4|\log \varepsilon|.  \]
  \end{lemma}
  \begin{proof}
  Consider the equation
  \begin{equation}\label{4-16}
  -\Delta_{\mathbb{H}^2} \omega= UV,~~\omega \in \mathscr{D}^{1}.
  \end{equation}
Since up to some constant, the fundamental solution of $-\Delta_{\mathbb{H}^n}$ with pole at the origin is given by
$\frac{1}{|\xi|^{Q-2}}$ (cf. Folland-Stein \cite{FS}), we obtain
  \begin{equation}\label{4-17}
  \omega (\eta)=\int_{\mathbb{H}^2 }  \frac{U(\xi)V(\xi)}{|\eta^{-1}\circ \xi|^{4}} \, d\xi.
  \end{equation} 
 Multiplying \eqref{4-16} by $\omega$ and integrating over $\mathbb{H}^2 $, we obtain
  $$\|UV\|_{\mathscr{D}^{-1}}^2 = \|\omega\|_{\mathscr{D}^{1}}^2 = \int_{\mathbb{H}^2 } \int_{\mathbb{H}^2 } \frac{UV(\xi)UV(\eta)}{|\eta^{-1}\circ \xi|^{4}} \, d\xi  d\eta.$$

 If $\varepsilon=\lambda$, we note that
$$V(\xi) \approx \frac{1}{( \frac{1}{\sqrt{\lambda}} +\sqrt{\lambda}|\xi_0^{-1}\circ \xi|  )^{4}} \le \lambda^{2}, ~~  |\xi^0|\leq \frac{1}{\lambda},$$
   we obtain
 $$V  \approx \lambda^{2} \mbox{ in } B_{\frac{2}{\lambda}},$$
 and
 $$ V(\xi) \approx  \lambda^{-2}|\xi_0^{-1}\circ \xi|^{-4} \approx  \lambda^{-2}| \xi|^{-4} \mbox{ in } B_{\frac{2}{\lambda}}^c. $$
It follows that
\begin{align*}
 \iint_{ B_{\frac{2}{\lambda}}\times  \mathbb{H}^2} \frac{UV(\xi)UV(\eta)}{|\eta^{-1}\circ \xi|^{4}} \, d\xi  d\eta
&\lesssim  \iint_{ B_{\frac{2}{\lambda}}\times  \mathbb{H}^2} \frac{\lambda^4}{|\eta^{-1}\circ \xi|^{4}(1+|\xi|^4)(1+|\eta|^4)}  \, d\xi  d\eta.
\end{align*}
We make the decomposition
$$\iint_{ B_{\frac{2}{\lambda}}\times  \mathbb{H}^2}=\int_{B_{\frac{2}{\lambda}}} \int_{\{|\xi|<\frac{|\eta|}{2}\}} +\int_{B_{\frac{2}{\lambda}}} \int_{\{\frac{|\eta|}{2} \le |\xi|\le 2|\eta|\}} +\int_{B_{\frac{2}{\lambda}}} \int_{\{|\xi|>2|\eta|\}} .$$
By direct calculation, we have
\begin{equation}\label{4-18}
\begin{aligned}
\int_{B_{\frac{2}{\lambda}}} \int_{\{|\xi|<\frac{|\eta|}{2}\}} & \frac{\lambda^4}{|\eta^{-1}\circ \xi|^{4}(1+|\xi|^4)(1+|\eta|^4)}  \, d\xi  d\eta \\
&\approx \int_{B_{\frac{2}{\lambda}}} \int_{\{|\xi|<\frac{|\eta|}{2}\}} \frac{\lambda^4}{|\eta|^{4}(1+|\xi|^4)(1+|\eta|^4)}  \, d\xi  d\eta \\
&\lesssim  \int_{B_{\frac{2}{\lambda}}} \frac{\lambda^4}{|\eta|^{2}(1+|\eta|^4)} d\eta\\
&\lesssim  \lambda^4 \log |\frac{1}{\lambda}|,
\end{aligned}
\end{equation}
\begin{equation}
\begin{aligned}
\int_{B_{\frac{2}{\lambda}}} \int_{\{\frac{|\eta|}{2}  \le |\xi|\le 2|\eta|\}} &\frac{\lambda^4}{|\eta^{-1}\circ \xi|^{4}(1+|\xi|^4)(1+|\eta|^4)}  \, d\xi  d\eta \\
&\lesssim \int_{B_{\frac{2}{\lambda}}} \int_{\{\frac{1}{2}  \le |\xi|\le 2\}} \frac{\lambda^4 }{|\eta|^2 |\tau_{\delta_{\frac{1}{|\eta|}} (\eta) }^{-1}( \xi)|^{4}  (1+|\eta|^4)}  \, d\xi  d\eta \\
&\lesssim \int_{B_{\frac{2}{\lambda}}} \frac{\lambda^4 }{|\eta|^2 ( 1+|\eta|^4 )}   d\eta \\
&\lesssim \lambda^4 \log |\frac{1}{\lambda}| ,
\end{aligned}
\end{equation}
and
\begin{equation}
\begin{aligned}
\int_{B_{\frac{2}{\lambda}}} \int_{\{|\xi|>2|\eta|\}} &\frac{\lambda^4}{|\eta^{-1}\circ \xi|^{4}(1+|\xi|^4)(1+|\eta|^4)}  \, d\xi  d\eta \\
&\approx \int_{B_{\frac{2}{\lambda}}} \int_{\{|\xi|>2|\eta| \}} \frac{\lambda^4}{|\xi|^{4}(1+|\xi|^4)(1+|\eta|^4)}  \, d\xi  d\eta \\
&\lesssim \int_{B_{\frac{2}{\lambda}}} \frac{\lambda^4}{|\eta|^{2} (1+|\eta|^4)} d\eta \\
&\lesssim \lambda^4 \log |\frac{1}{\lambda}| .
\end{aligned}
\end{equation}
On the other hand, in the region $B_{\frac{2}{\lambda}}^c \times B_{\frac{2}{\lambda}}^c $, we have
\begin{equation}
\begin{aligned}
\iint_{ B_{\frac{2}{\lambda}}^c\times  B_{\frac{2}{\lambda}}^c} \frac{UV(\xi)UV(\eta)}{|\eta^{-1}\circ \xi|^{4}} \, d\xi  d\eta
\lesssim  \iint_{ B_{\frac{2}{\lambda}}^c\times  B_{\frac{2}{\lambda}}^c} \frac{1}{\lambda^4 |\eta^{-1}\circ \xi|^{4}|\xi|^8|\eta|^8}  \, d\xi  d\eta.
\end{aligned}
\end{equation}
Through direct computation, we obtain
\begin{equation}
\begin{aligned}
\int_{B_{\frac{2}{\lambda}}^c} \int_{\{\frac{2}{\lambda} \le |\xi|\le \frac{|\eta|}{2}\}} &\frac{1}{\lambda^4 |\eta^{-1}\circ \xi|^{4}|\xi|^8|\eta|^8}  \, d\xi  d\eta \\
&\approx \int_{B_{\frac{4}{\lambda}}^c} \int_{\{\frac{2}{\lambda} \le |\xi|\le \frac{|\eta|}{2}\}} \frac{1}{\lambda^4 |\eta |^{4}|\xi|^8|\eta|^8}  \, d\xi  d\eta \\
&\approx \int_{\frac{4}{\lambda}}^{+\infty} \frac{(\frac{\lambda^2}{4}-\frac{4}{r^2 })r^5 }{\lambda^4 r^{12}} dr \lesssim \lambda^4 ,
\end{aligned}
\end{equation}
\begin{equation}
\begin{aligned}
\int_{B_{\frac{2}{\lambda}}^c} \int_{\frac{|\eta|}{2}\le |\xi|\le 2|\eta| } &\frac{1}{\lambda^4 |\eta^{-1}\circ \xi|^{4}|\xi|^8|\eta|^8}  \, d\xi  d\eta\\
&\approx \int_{B_{\frac{2}{\lambda}}^c} \int_{\{\frac{1}{2} \le |\xi|\le 2 \}} \frac{1}{\lambda^4 |\eta|^{14} |\tau_{\delta_{\frac{1}{|\eta|}} (\eta) }^{-1}( \xi)|^{4} }  \, d\xi  d\eta \\
&\lesssim \int_{B_{\frac{2}{\lambda}}^c} \frac{1}{\lambda^4 |\eta|^{14} }    d\eta \lesssim \lambda^4 ,
\end{aligned}
\end{equation}
and
\begin{equation}\label{4-24}
\begin{aligned}
\int_{B_{\frac{2}{\lambda}}^c} \int_{\{|\xi|\ge 2|\eta| \}}& \frac{1}{\lambda^4 |\eta^{-1}\circ \xi|^{4}|\xi|^8|\eta|^8}  \, d\xi  d\eta \\
&\approx \int_{B_{\frac{2}{\lambda}}^c} \int_{\{  |\xi|\ge 2|\eta|\}} \frac{1}{\lambda^4 |\xi|^{12}|\eta|^8}  \, d\xi  d\eta \\
&\approx \int_{B_{\frac{2}{\lambda}}^c} \int_{\{  |\xi|\ge 2|\eta|\}} \frac{1}{\lambda^4 |\eta|^{14} } d\eta \lesssim \lambda^4 .
\end{aligned}
\end{equation}
Thus, by \eqref{4-18}-\eqref{4-24}, we obtain
\begin{equation}\label{4-25}
 \int_{\mathbb{H}^2 } \int_{\mathbb{H}^2 } \frac{UV(\xi)UV(\eta)}{|\eta^{-1}\circ \xi|^{4}} \, d\xi  d\eta\lesssim \lambda^4|\log \lambda|.
\end{equation}

If $\varepsilon= \frac{1}{\lambda |\xi_0|^2}$, then $|\xi^0|\geq \frac{1}{\lambda}$ . We have
 \[ V(\xi)\approx \frac{\lambda^{2}}{1+\lambda^{4}|\xi_0^{-1}\circ \xi|^{4} }, ~U\approx |\xi_0|^{-4}\mbox{ in } B_{\frac{|\xi_0|}{2}}(\xi_0),  \]
 and
 \[  V(\xi) \approx \lambda^{-2}|\xi_0^{-1}\circ \xi|^{-4} \mbox{ in } B_{\frac{|\xi_0|}{2}}^c(\xi_0). \]
 As a consequence, we obtain
 \begin{equation}\label{4-26}
 \begin{aligned}
 \int_{\mathbb{H}^2 } \frac{UV(\xi)}{|\eta^{-1}\circ \xi|^{4}} \, d\xi =& (\int_{B_{\frac{|\xi_0|}{2}}(\xi_0)} +\int_{ B_{\frac{|\xi_0|}{2}}^c(\xi_0)} )\frac{UV(\xi)}{|\eta^{-1}\circ \xi|^{4}} \, d\xi
 \\\approx&
 \int_{B_{\frac{|\xi_0|}{2}}(\xi_0)}  |\xi_0|^{-4} \frac{\lambda^{2}}{1+\lambda^{4}|\xi_0^{-1}\circ \xi|^{4}}\frac{1}{|\eta^{-1}\circ \xi|^{4}} \, d\xi
 \\&+\int_{ B_{\frac{|\xi_0|}{2}}^c(\xi_0)} \frac{\lambda^{-2}|\xi_0^{-1}\circ \xi|^{-4}}{(1+|\xi|^4)|\eta^{-1}\circ \xi|^{4}} \, d\xi .
 \end{aligned}
 \end{equation}
 If 
 $\eta \in B_{\frac{2|\xi_0 |}{3} }  (\xi_0 )$, 
be setting $\zeta=\xi_0^{-1}\circ \eta$, we obtain

 \begin{equation}
 \begin{aligned}
 \int_{B_{\frac{|\xi_0|}{2}}(\xi_0) }  & \frac{\lambda^{2}}{1+\lambda^{4}|\xi_0^{-1}\circ \xi|^{4}}\frac{1}{|\eta^{-1}\circ \xi|^{4}}\, d\xi 
 = \int_{B_{\frac{|\xi_0|}{2}}}  \frac{\lambda^{2}}{1+\lambda^{4}|\xi|^{4}}\frac{1}{|\zeta^{-1}\circ \xi|^{4}}  \, d\xi \\&
 \lesssim \int_{\{|\xi|\le \frac{|\zeta|}{2} \}} \frac{\lambda^{2}}{1+\lambda^{4} |\xi|^{4}}\frac{1}{|\zeta|^{4}}\, d\xi
  +\int_{\{2|\zeta|\le |\xi|\}}\frac{\lambda^2}{1+\lambda^{4} |\xi|^{4}}\frac{1}{|\xi|^{4}}  \, d\xi \\&
  + \int_{\{\frac{1}{2}\le |\xi|\le 2\}}  \frac{\lambda^{2}|\zeta|^2 }{  (1+\lambda^{4} |\zeta|^{4})}\frac{1}{| \tau_{\delta_{\frac{1}{|\zeta|}}(\zeta)}^{-1}(\xi)  |^{4}} \, d\xi\\
 &
 \approx \frac{1  }{\lambda^4  |\zeta|^4 } \int_0^{\frac{\lambda|\zeta|}{2}}  \frac{ r^5 }{1+ r^4 }  d r + \int_{2\lambda|\zeta |}^{+\infty}\frac{r}{1+ r^4 }  d  r +\frac{\lambda^2 |\zeta|^2 }{ 1+\lambda^4 |\zeta|^4 }  \\&
 \lesssim  \frac{1}{ 1+\lambda^2 |\xi_0^{-1} \circ \eta|^2 } .
 \end{aligned}
 \end{equation}
If $\eta\in B_{\frac{2 |\xi_0|}{3}}^c(\xi_0) $, then 
\begin{equation}
\begin{aligned}
\int_{B_{\frac{|\xi_0|}{2}}(\xi_0)} \frac{\lambda^{2}}{1+\lambda^{4}|\xi_0^{-1}\circ \xi|^{4}}\frac{1}{|\eta^{-1}\circ \xi|^{4}} \, d\xi
&\approx \int_{B_{\frac{|\xi_0|}{2}}}  \frac{\lambda^{2}}{1+\lambda^{4} |\xi|^{4}}\frac{1}{|\zeta|^{4}}\, d\xi \\&
 \approx \frac{|\xi_0|^2}{\lambda^2|\xi_0^{-1}\circ \eta|^4}.
\end{aligned}
\end{equation}
It follows that
  \begin{equation}\label{4-29}
 \begin{aligned}
 \int_{B_{\frac{|\xi_0|}{2}}(\xi_0)} & |\xi_0|^{-4} \frac{\lambda^{2}}{1+\lambda^{4}|\xi_0^{-1}\circ \xi|^{4}}\frac{1}{|\eta^{-1}\circ \xi|^{4}} \, d\xi \\
 \lesssim &\frac{\chi_{\{ \eta\in B_{\frac{2 |\xi_0|}{3}}(\xi_0) \}} }{|\xi_0|^4(1+\lambda^2|\xi_0^{-1}\circ \eta|^2)} +\frac{\chi_{\{ \eta\in B_{\frac{2 |\xi_0|}{3}}^c(\xi_0) \}}}{\lambda^2|\xi_0^{-1}\circ \eta|^4|\xi_0|^2}.
 \end{aligned}
 \end{equation}
On the other hand, we can deduce that
  \begin{equation}\label{4-32}
 \begin{aligned}
\int_{ B_{\frac{|\xi_0|}{2}}^c(\xi_0)} & \frac{\lambda^{-2}|\xi_0^{-1}\circ \xi|^{-4}}{(1+|\xi|^4)|\eta^{-1}\circ \xi|^{4}} \, d\xi \\
&\lesssim \lambda^{-2}|\xi_0|^{-4} \int_{\mathbb{H}^2} \frac{1}{(1+|\xi|^4)|\eta^{-1}\circ \xi|^{4}} \, d\xi\\&
\lesssim \frac{1}{\lambda^2  |\xi_0|^4(1+|\eta|^2) }.
 \end{aligned}
 \end{equation}
Thus, \eqref{4-26}, \eqref{4-29} and \eqref{4-32} yield that
 \[  \int_{\mathbb{H}^2 }\frac{UV(\xi)}{|\eta^{-1}\circ \xi|^{4}} \, d\xi \lesssim
  \frac{\chi_{\{ \eta\in B_{\frac{2 |\xi_0|}{3}}(\xi_0) \}}}{|\xi_0|^4(1+\lambda^2|\xi_0^{-1}\circ \eta|^2)}
  +\frac{\chi_{\{ \eta\in B_{\frac{2|\xi_0|}{3}}^c(\xi_0) \}}}{\lambda^2|\xi_0|^4 (1+|\eta|^2) }
.
  \]
We conclude that
  \begin{equation}\label{4-33}
  \begin{aligned}
  \iint_{\mathbb{H}^2 \times \mathbb{H}^2 } & \frac{UV(\xi)UV(\eta)}{|\eta^{-1}\circ \xi|^{4}} \, d\xi  d\eta \\
  \lesssim &
  \int_{\{ \eta\in B_{\frac{2 |\xi_0|}{3}}(\xi_0) \}} \frac{\lambda^2}{{|\xi_0|^8(1+\lambda^6|\xi_0^{-1}\circ \eta|^6)}} d\eta\\
  &+\int_{\{ \eta\in B_{\frac{2|\xi_0|}{3}}^c(\xi_0) \cap B_{|\xi_0|}\}} \frac{1}{\lambda^4|\xi_0|^8 (1+|\eta|^6) } d\eta \\
  &+\int_{\{ \eta\in B_{\frac{2|\xi_0|}{3}}^c(\xi_0) \cap B_{|\xi_0|}^c\}}\frac{1}{\lambda^4|\xi_0|^4|\eta|^{10}  } d\eta\\
  \lesssim & \frac{\log|\lambda|\xi_0|^2|}{\lambda^4|\xi_0|^8}.
  \end{aligned}
  \end{equation}
  Therefore, the conclusion follows from \eqref{4-25} and \eqref{4-33}.
   \end{proof}

\section*{Acknowledgements}
Hua Chen is supported by National Natural Science Foundation of China (Grant Nos. 12131017, 12221001) and National Key R\&D Program of China (no. 2022YFA1005602).

 \section*{Declarations}
On behalf of all authors, the corresponding author states that there is no conflict of interest. Our manuscript has no associated data.

\end{document}